\documentclass[oneside,a4paper,12pt]{article}
\usepackage[english]{babel}
\usepackage[utf8]{inputenc}
\usepackage[T1]{fontenc}
\usepackage[top=25mm, bottom=25mm, left=25mm, right=25mm]{geometry}
\usepackage{booktabs}
\usepackage{color}
\usepackage{hyperref}
\usepackage{graphicx}
\usepackage{amsmath}
\usepackage{amssymb}
\usepackage{amsthm}
\usepackage{amsfonts}
\usepackage{float}
\graphicspath{{./Figuras/}}    
\definecolor{shadecolor}{rgb}{0.8,0.8,0.8}
\newcommand{\Ric}{\mbox{Ric}}

\newcommand{\tr}{\mbox{tr}}
\newcommand{\spa}{\mbox{span}}
\newcommand{\Hess}{\mbox{Hess}}
\newcommand{\End}{\mbox{End}}
\newcommand{\grad}{\mbox{grad}}
\newcommand{\Id}{\mbox{Id}}
\newcommand{\TP}{\mbox{TP}}
\newtheorem{theorem}{Theorem}[section]
\newtheorem{proposition}{Proposition}[section]
\newtheorem{corollary}{Corollary}[section]

\newtheorem{lemma}{Lemma}[section]

\newtheorem{Remarks}{Remarks}[section]
\newtheorem{observation}{Remark}[section]
\newtheorem{example}{Example}[section]
\newtheorem{examples}{Examples}[section]

\begin{document}

\title{Submanifolds with constant Moebius curvature and flat normal bundle}
\maketitle
\begin{center}
\author{M. S. R. Antas        and
        R. Tojeiro$^*$}  
        \footnote{Corresponding author}
      \footnote{The first author was supported by FAPESP Grant 2019/04027-7. The second author is partially supported by FAPESP grant 2022/16097-2 and CNPq grant 307016/2021-8.\\
      Data availability statement: Not applicable.}
\end{center}
\date{}

\begin{abstract}
We classify isometric immersions $f:M^{n}\to \mathbb{R}^{n+p}$, $n \geq 5$ and $2p \le n$, with constant Moebius curvature and flat normal bundle. 
\end{abstract}

\vspace{2ex}


\date{}
\maketitle

\section{Introduction}

 In a seminal paper in Moebius Submanifold Geometry, C. P. Wang \cite{Wang} introduced a  Moebius invariant metric $g^*$ on an umbilic-free hypersurface $f\colon M^n \to \mathbb{R}^{n+1}$, called the \emph{Moebius metric},  and a Moebius invariant $2$-form $B$ on $M^n$, the \emph{Moebius second fundamental form} of $f$, and  proved that, for $n\geq 4$, the pair $(g^*,B)$ forms a complete Moebius invariant system which determines the hypersurface up to Moebius transformations (see also Theorem $9.22$ in \cite{DT} for a general Moebius fundamental theorem for submanifolds of arbitrary dimension and codimension). The corresponding conformal Gauss, Codazzi and Ricci equations involve two other important Moebius invariant tensors named the \emph{Blaschke tensor} and the \emph{Moebius form}.
 
   C. P. Wang's work motivated several authors to investigate umbilic-free immersions  whose associated Moebius invariants have a simple structure or particular natural properties (see, among others,  \cite{surfaces-vanishing-moebius-form}, \cite{Rodrigues-Keti},  \cite{isopara-three-principals},  \cite{CMF-surfaces}, \cite{moebius-ricci-curvature}, \cite{Moebius-Laguerre}, \cite{parallel-moebius-second-fund-form} and \cite{CMF-three-principals}).

   Among the relevant classes of  umbilic-free immersions $f:M^{n}\to \mathbb{R}^{n+p}$ is that of submanifolds with  \emph{constant Moebius curvature}, meaning that  $(M^{n}, g^{*})$ has constant sectional curvatures.  A local classification of the umbilic-free  hypersurfaces $f:M^{n}\to \mathbb{R}^{n+1}$, $n \geq 4$, which have constant Moebius curvature was first obtained in \cite{classif-moebius-hyper}, and a somewhat simpler proof  was subsequently given in \cite{deform}. The classification also applies to the case in which $n=3$ and $f$ is assumed to have a principal curvature of multiplicity $2$.
   
   The starting point for achieving such classification is the fact that every umbilic-free hypersurface with constant Moebius curvature is conformally flat. Recall that a Riemannian manifold $(M^{n}, g)$ 
 is \emph{conformally flat} if every point of  $M^{n}$ has an open  neighborhood that is conformal to an open subset of the Euclidean space  $\mathbb{R}^{n}.$  In particular, by a well-known result due to E. Cartan \cite{Cartan}, an umbilic-free {conformally flat hypersurface $f:M^{n}\to \mathbb{R}^{n+1}$,  $n \geq 4$, must have a principal curvature $\lambda$ of multiplicity $n-1$. It is not difficult to show that this implies the Moebius form of $f$ to be closed.  
 
     These facts enabled the 
authors, by means of the method of moving frames, to reduce the problem to proving that a conformally flat Euclidean hypersurface with closed Moebius form must be locally Moebius congruent to either a  cylinder over a curve $\gamma:I \to \mathbb{R}^{2}$, a cylinder over a surface of $\mathbb{R}^3$ which is itself a cone over a curve $\gamma:I \to \mathbb{S}^{2}$, or a rotation hypersurface over a curve $\gamma:I \to \mathbb{R}_{+}^{2}$, where $\mathbb{R}^2_{+}$ is regarded as the Poincar\'e half-plane model of the hyperbolic plane. Then, imposing the sectional curvatures of the Moebius metric to be constant has led to an explicit expression for the curvature of $\gamma$, which was called a \emph{curvature spiral}.
 
    In this article we initiate the investigation of umbilic-free immersions $f:M^{n}\to \mathbb{R}^{n+p}$,  $n\geq 5$, with constant  Moebius curvature and higher codimension. We restrict ourselves in this paper to the case in which the submanifold has flat normal bundle.  

     Our main result is a classification of the umbilic-free imersions $f:M^{n}\to \mathbb{R}^{n+p}$ with constant  Moebius curvature and flat normal bundle for which $n \geq 5$  and $2p \le n$. See Theorem \ref{theorem-main} for the precise statement. As an important first step, for any umbilic-free imersion $f:M^{n}\to \mathbb{R}^{n+p}$ of a conformally flat manifold $M^n$, we determine the intrinsic local structure of $M^n$ endowed with its Moebius metric.   We also classify, with no restrictions on the codimension, the isometric immersions $f:M^{n}\to \mathbb{R}^{n+p}$, $n\geq 3$,  with constant  Moebius curvature and flat normal bundle that have exactly two distinct principal normal vector fields. The latter result already  extends (and substantially simplifies the proofs of) the classifications in both \cite{classif-moebius-hyper}  and \cite{deform} of the umbilic-free hypersurfaces with constant Moebius curvature.  

\section{Preliminaries}\label{preli}

   In this section we collect several definitions and known results that will be needed in the proofs of our results.

   \subsection{Principal normal vector fields}

A normal vector $\eta \in N_{f}M(x)$ is said to be a \emph{principal normal vector with multiplicity $s$} of an isometric immersion $f:M^{n}\to \Tilde{M}^{m}$ at $x \in M^{n}$  if the vector subspace $$ E_{\eta}(x)=\{X\in T_{x}M: \alpha^{f}(X,Y)=\langle X,Y\rangle \eta \,\,\, \,\, \mbox{for all} \,\, Y \in T_{x}M\}$$ has dimension $s>0,$ where $\alpha^{f}$ denotes the second fundamental form of $f$. 
A normal vector field $\eta \in \Gamma(N_{f}M)$ is called a \emph{principal normal vector field} of $f$ with multiplicity $s$ if $\dim E_{\eta}(x)=s$ for all $x\in M^{n}$, in which case $E_{\eta}$ is a smooth distribution.
The normal vector field $\eta \in \Gamma(N_{f}M)$ is said to be \emph{Dupin} if $\nabla^\perp_T\eta=0$ for all $T\in \Gamma(E_{\eta})$. 
If, in particular, $\eta$ is identically zero, then $E_{\eta}(x)$ is the kernel of $\alpha_f$, called  the \emph{relative nullity} subspace of $f$ at $x$ and $\nu(x):=\dim E_{\eta}(x)$ the \emph{index of relative nullity} of $f$ at $x.$

    A smooth distribution $E$ on a Riemannian manifold $M^{n}$ is \emph{umbilical} if there exists a smooth section $\delta$ of $E^{\perp}$, named the \emph{mean curvature vector field} of $E$, such that $$\langle \nabla_{T}S, X\rangle=\langle T,S\rangle \langle \delta, X\rangle$$ for all $T,S \in \Gamma(E)$ and $X \in \Gamma(E^{\perp}).$ If $\delta$ is identically zero, then $E$ is said to be \emph{totally geodesic}. 
    An umbilical distribution on $M^{n}$ is always integrable and its leaves are umbilical submanifolds of $M^{n}.$ If, in addition, $ (\nabla_{T}\delta)_{E^{\perp}}=0$ for all $T \in \Gamma(E)$, then $E$ is said to be \emph{spherical}, and its leaves are called \emph{extrinsic spheres} of $M^n$.

   The following fact is well-known (see, e.g., Proposition $1.22$ of \cite{DT}).
\begin{proposition}[\cite{Re}]\label{dupin}
    \textit{Let $f:M^{n}\to \mathbb{Q}^{m}_{c}$ be an isometric immersion with a principal normal vector field $\eta$ of multiplicity $s>0$. Then the following assertions hold:
    \begin{description}
        \item[(i)] If $s\geq 2$, then $\eta$ is Dupin.
        \item[(ii)] The principal normal vector field $\eta$ is Dupin if and only if $E_{\eta}$ is a spherical distribution and $f$ maps each leaf of $E_{\eta}$ into an extrinsic sphere of $\mathbb{Q}^{m}_{c}.$
    \end{description}}
    
\end{proposition}

  A key result in the proof of our main result in this article is the following.

  \begin{theorem}[\cite{DFT}]\label{cor:dupinpcn} Let $f\colon M^n\to\mathbb{R}^m$ be an 
isometric immersion that carries a Dupin principal normal vector field 
$\eta$ with multiplicity $k$. Assume that  $E_\eta^\perp$ is an umbilical 
distribution. If $k=n-1$, suppose further that the integral curves of 
$E_\eta^\perp$ are extrinsic circles of $M^n$. 
Then $f(M^{n})$ is, up to a conformal transformation of $\mathbb{R}^m$, an 
open subset of a submanifold of one of the following types:
\begin{itemize}
\item[(i)] A $k$-cylinder over an isometric immersion $g\colon M^{n-k}\to\mathbb{R}^{m-k}$;
\item[(ii)] A $(k-1)$-cylinder over an isometric immersion 
$G\colon M^{n-k+1}\to\mathbb{R}^{m-k+1}$ which is itself 
a cone over an isometric immersion $g\colon M^{n-k}\to\mathbb{S}^{m-k}$;
\item[(iii)] A rotation submanifold over an isometric immersion 
$h\colon M^{n-k}\to\mathbb{R}^{m-k}$.
\end{itemize}
\end{theorem}

  The immersion $f$ in part $(ii)$ (respectively, part $(iii)$) of the preceding theorem can be alternatively described in terms of the conformal diffeomorphism   $\Theta:\mathbb{S}^{m-k}\times \mathbb{H}^{k}\to \mathbb{R}^{m}$ (respectively, $\Theta:\mathbb{S}^{k}\times \mathbb{H}^{m-k}\to \mathbb{R}^{m}$) defined by $ \Theta(y,z)=(z_1y, z_2, \ldots, z_{k})$  (respectively, $ \Theta(y,z)=(z_1y, z_2, \ldots, z_{m-k})$).
Namely, $M^n$ splits as $M^n=M^{n-k}\times \mathbb{H}^{k}$ (respectively, $M^n=\mathbb{S}^{k}\times M^{n-k}$) and $f=\Theta\circ (g\times Id)$ (respectively, $f=\Theta\circ (Id\times h)$), where $ \mathbb{H}^{k}$ is given by its half-space model.

\subsection{Conformally flat manifolds and submanifolds}

A Riemannian manifold $M^{n}$ is said to be \emph{conformally flat} if each point $x$ has a neighborhood which is conformally diffeomorphic to an open subset of Euclidean space $\mathbb{R}^{n}$. In particular, every Riemannian manifold with constant sectional curvature is conformally flat. We will need the following well-known characterization of conformally flat Riemannian products. 

\begin{proposition}\label{cf-products}
    A Riemannian product $(M_1, g_1)\times (M_2, g_2)=(M_1\times M_2, g_1 + g_2)$ of dimension $n \geq 3$ is conformally flat if and only if one of the following possibilities hold \begin{description}
    \item[i)] $(M_i,g_i)$ \textit{is $1$-dimensional and $(M_j, g_j)$, $i \ne j$, has constant sectional curvature.
    \item[ii)]$(M_1,g_1)$ and $(M_2,g_2)$ are either both flat or are  manifolds with dimension at least two of constant sectional curvatures  with the same absolute values and opposite signs}.
\end{description}
\end{proposition}

   We will also make use of the following extension due to J. D. Moore (\cite{Moore}) of    E. Cartan's theorem on Euclidean conformally flat hypersurfaces of dimension $n\geq 4$. 

\begin{theorem}\label{theorem-moore}
	Let $f:M^{n}\to \mathbb{R}^{n+p}$, $n \geq 4$, be an isometric immersion of a conformally flat manifold. If $p \le n-3$, then for each $x \in M^{n}$ there exists a principal normal vector $\eta \in N_{f}M(x)$ such that $\dim \, E_{\eta}(x) \geq n-p \geq 3.$  
\end{theorem}

Isometric immersions with flat normal bundle $f:M^{n}\to \mathbb{R}^{n+p}$, $n \geq 4$, of a conformally flat manifold have been recently studied by Dajczer, Onti and Vlachos \cite{DOV}, who in  particular proved the following fact.

\begin{theorem}\label{theorem-DOV}
		A proper isometric immersion $f:M^{n} \to \mathbb{R}^{n+p}$, $n \geq 4$, with flat normal bundle of a conformally flat manifold can admit at  most one principal normal vector field of multiplicity greater than one.
  \end{theorem}
  
  \subsection{Twisted products}
  
      Let $M=\Pi_{i=0}^{k}M_i$ 
      be a product of smooth manifolds $M_0, \ldots, M_k$. A metric $\langle \, , \,\rangle$ on $M^{n}$ is called a \emph{twisted product metric} if there exist Riemannian metrics $\langle \,,\,\rangle_i$ on $M_i$ and smooth maps $\rho_i:M\to \mathbb{R}_{+}$, $0 \le i \le k,$ such that $$\langle \,,\,\rangle=\sum_{i=0}^{k}\rho_i^{2}\pi_i^{*}\langle \,,\,\rangle_{i},$$ where $\pi_i:M \to M_i$ denotes the canonical projection. Then $(M, \langle \,,\,\rangle)$ is said to be a \emph{twisted product} and is denoted by ${^{\rho}\prod_{i=0}^{k}}(M_i, \langle \,,\,\rangle_i)$, where $\rho=(\rho_0, \ldots, \rho_k)$.  When $\rho_1, \ldots, \rho_k$ are independent of $M_1, \ldots, M_k$, that is, there exist $\Tilde{\rho}_a \in C^{\infty}(M_0)$ such that $\rho_a=\Tilde{\rho}_a \circ \pi_0$ for $a=1, \ldots, k$ and, in addition, $\rho_0$ is identically $1$, then $\langle \,,\,\rangle$ is called a \emph{warped product metric} and $(M, \langle \,,\,\rangle):=(M_0, \langle \,,\,\rangle_0)\times _{\rho}\prod_{a=1}^{k}(M_a, \langle \,,\,\rangle_a)$ a \emph{warped product} with \emph{warping function} $\rho=(\rho_1, \ldots, \rho_k).$ If $\rho_i$ is identically $1$ for $i=0, \ldots, k,$ the metric $\langle \,,\,\rangle$ is a usual \emph{Riemannian product metric}, in which case $(M, \langle \,,\,\rangle)$ is called a \emph{Riemannian product}.
      
      A \emph{net} $\mathcal{E}=(E_i)_{i=0}^{r}$ on a differentiable manifold $M$ is a splitting $TM=\oplus_{i=1}^{r}E_i$ of its tangent bundle into a family of integrable distributions. A net $\mathcal{E}=(E_i)_{i=1}^{r}$ on a Riemannian manifold $M$ is called an \emph{orthogonal net} if the distributions of $\mathcal{E}$ are mutually orthogonal.

 Let $M$ be a product manifold. The \emph{product net} of $M$, $\mathcal{E}=(E_i)_{i=0}^{k}$, is defined by $$E_i(x)={\tau_i^{x}}_{*}T_{x_i}M_i, \quad 0 \le i \le k,$$ for any $x=(x_0, \ldots, x_{k})\in M,$ where $\tau_i^{x}:M_i \to M$ is the inclusion of $M_i$ into $M$ given by $$\tau_{i}^{x}(\Bar{x}_i)=(x_0, \ldots, \Bar{x}_i, \ldots, x_k), \quad 0 \le i \le k.$$

  The relation between the Levi-Civita connections of a twisted-product metric and of the corresponding Riemannian product metric is as follows (see \cite{MHM}).

 \begin{proposition}[\cite{MHM}]
Let $(M, \langle\,,\,\rangle)={^{\rho}\prod_{i=0}^{k}}(M_i, \langle \,,\,\rangle_i)$ be a twisted product with twist function  $\rho=(\rho_0, \ldots, \rho_k)$ and  product net $\mathcal{E}=(E_i)_{i=0}^{k}$. Let $\nabla$ and $\tilde{\nabla}$ be the Levi-Civita connections of $\langle\,,\,\rangle $ and of the product metric $\langle \,,\,\tilde{\rangle}$, respectively, and let $U_i=-\grad\,(\log \circ \rho_i)$, $0 \le i \le k$, where the gradient is calculated with respect to $\langle \,,\,\rangle.$ Then \begin{equation}\label{twisted-riemannian}
    \nabla_{X}Y=\tilde{\nabla}_{X}Y+\sum_{i=0}^{k}(\langle X^{i},Y^{i} \rangle U_i-\langle X, U_i \rangle Y^{i}-\langle Y, U_i\rangle X^{i}),
\end{equation} where $X \mapsto X^{i}$ is the orthogonal projection over $E_i.$
\end{proposition}
It follows from \eqref{twisted-riemannian} that
 \begin{equation}\label{umbilical} (\nabla_{X_i}Y_i)_{E_i^{\perp}}=\langle X_i,Y_i \rangle (U_i)_{E_i^{\perp}}.\end{equation} 
 for all $X_i,Y_i \in E_i$. Therefore, $E_i$ is an umbilical distribution with mean curvature vector field $(U_i)_{E_i^{\perp}}$. It also follows immediately from \eqref{twisted-riemannian} that $E_{i}^{\perp}=\oplus_{\overset{j=0}{j \ne i}}^{k}E_{j}$ is integrable for all $0 \le i \le k$.

An orthogonal net $\mathcal{E}=({E}_i)_{i=0}^{k}$ on a Riemannian manifold $M$ is called a $\TP$-net if $E_i$ is umbilical and $E_i^{\perp}$ is integrable for all $i=0, \ldots, k$.
In particular, the product net of a twisted product is a $\TP$-net.

Let $\mathcal{F}=\{F_i\}_{i=0}^{k}$ be a net on a smooth manifold $N$. Then $\bar{\Phi}\colon M:=\prod_{i=0}^{k}M_i \to N$ is called a \emph{product representation} of $\mathcal{F}$ if $\bar{\Phi}$ is a diffeomorphism and $\bar{\Phi}_{*}E_i(x)=F_i(\bar{\Phi}(x))$ for each $x \in M$ and each $i=0, \ldots, k,$ where  $\mathcal{E}=(E_i)_{i=0}^{k}$ is the product net of $M$.

   The following de Rham-type decomposition theorem was proved in \cite{MHM}.

\begin{theorem}\label{decomposition-theorem}
\textit{Let $\mathcal{E}=(E_i)_{i=0}^{k}$ be a $\TP$-net on a Riemannian manifold $M$. Then, for each $x \in M$, there exists  a local product representation $\Bar{\Phi}:\prod_{i=0}^{k}M_i \to U$ of $\mathcal{E}$, with $x \in U \subset M$, which is an isometry with respect to a twisted product metric on $\prod_{i=0}^{k}M_i.$}
\end{theorem}

\subsection{Moebius invariants for Euclidean submanifolds}\label{moebius-invariants}

Let $\mathbb{L}^{m+2}$ be the \emph{Minkowski space} of dimension $m+2$, that is, $\mathbb{R}^{m+2}$  endowed  with the inner-product $$ \langle v,w\rangle=-v_0w_0+v_1w_1+\ldots+v_{m+1}w_{m+1}$$
 for all $v=(v_0, \ldots, v_{m+1})$, $w=(w_0, \ldots, w_{m+1})\in \mathbb{R}^{m+2}$. 
The \emph{light cone} $\mathbb{V}^{m+1}$ of $\mathbb{L}^{m+2}$ is the upper half of 
  $$\{v \in \mathbb{L}^{m+2}: \langle v, v\rangle =0 \,\,\, \mbox{and} \,\,\, v \ne 0\},$$ 
restricted to which the inner-product of $\mathbb{L}^{m+2}$ is degenerate. Its subset
\begin{equation*}
    \mathbb{E}^{m}=\mathbb{E}^{m}_{w}=\{p \in \mathbb{V}^{m+1}: \langle p, w\rangle=1\},
\end{equation*} gives rise to a  model of the $m$-dimensional Euclidean space $\mathbb{R}^{m}$ for each  fixed $w\in \mathbb{V}^{m+1}$. Indeed,  for any  $p_0 \in \mathbb{E}^{m}$ and any linear isometry $C:\mathbb{R}^{m}\to (\spa \{p_0, w\})^{\perp}\subset \mathbb{L}^{m+2}$, the  map $\Psi=\Psi_{p_0, w, C}:\mathbb{R}^{m}\to \mathbb{L}^{m+2}$, given by \begin{equation}\Psi(x)=p_0+Cx-\frac{1}{2}||x||^{2}w,\end{equation} is an  isometric embedding such that $\Psi(\mathbb{R}^{m})=\mathbb{E}^{m}.$ The position vector field $\Psi$ is a light-like parallel normal vector field along $\Psi$, whose second fundamental form is $$\alpha^{\Psi}(X,Y)=-\langle X,Y\rangle w$$
for all $X,Y \in \mathfrak{X}(\mathbb{R}^{m})$.

  Now let $f:M^{n}\to \mathbb{R}^{m}$ be an immersion free of umbilical points. Then, the function $$\rho^{2}=\frac{n}{n-1}(||\alpha^{f}||^{2}-n||\mathcal{H}^{f}||^{2})$$ does not vanish on $M^{n}$, where $\mathcal{H}^{f}$ and $||\alpha^{f}||$ stand for the mean curvature vector field of $f$ and the norm of the second fundamental form of $f$, respectively. The \emph{Moebius lift} of $f$ is the immersion $F:M \to \mathbb{V}^{m+1}\subset \mathbb{L}^{m+2}$ given by $F=\rho \, \Psi \circ f$, whose induced  metric  is \begin{equation}\label{moebius-metric}\langle \,,\,\rangle^{*}=\rho^{2}\langle \,,\,\rangle_{f}.\end{equation} The metric \eqref{moebius-metric} is called the \emph{Moebius metric} determined by $f$. It was proved in \cite{Wang} that the Moebius metric is invariant under conformal transformations of $\mathbb{R}^{m}.$

\emph{The Moebius second fundamental form} of $f:M^{n}\to \mathbb{R}^{m}$ is the symmetric section $\beta=\beta^{f} \in \mbox{Hom}^{2}(TM, TM; N_{f}M)$ defined by $$\beta(X,Y)=\rho (\alpha^{f}(X,Y)-\langle X,Y\rangle \mathcal{H}^{f})$$ 
 for all  $ X,Y \in \mathfrak{X}(M)$. Notice that $\beta$ is traceless and that its norm  with respect to  $\langle \,,\,\rangle^{*}$ is $||\beta||_{*}=\sqrt{\frac{n-1}{n}}$. 
  Associated with $\beta$ one defines the  \emph{Moebius third fundamental form } $III_{\beta} : \mathfrak{X}(M)\times \mathfrak{X}(M)\to \mathbb{R}$  by \begin{equation}\label{thirdform}
	III_{\beta}(X,Y)=\sum_{i=1}^{n}\langle \beta(X,X_i), \beta(Y,X_i)\rangle,
\end{equation} where  $X_1, \ldots, X_n$ is an orthonormal frame with respect to $\langle \,,\,\rangle^{*}.$ 

The \emph{Blaschke tensor} $\psi=\psi^{f}$ of $f$ is the symmetric $C^{\infty}(M)$-bilinear form given by $$\psi(X,Y)=\frac{1}{\rho} \langle \beta^{f}(X,Y), \mathcal{H}^{f} \rangle+\frac{1}{2\rho^{2}}(||\grad\,^{*}  \rho||^{2}_{*}+||\mathcal{H}^{f}||^{2})\langle X,Y\rangle^{*} - \frac{1}{\rho} \Hess\,^{*}  \rho(X,Y)$$
and its  \emph{Moebius form} $\omega=\omega^{f}$  is the normal bundle valued one-form defined by $$\omega(X)=-\frac{1}{\rho}(\nabla^{\perp}_{X}\mathcal{H}^{f}+\beta(X, \grad\, ^{*}\rho)),$$ where $\grad^{*}$ and $\Hess^{*}$ denote the gradient and the Hessian on $(M^{n}, \langle\,,\,\rangle^{*}).$

\begin{proposition}[\cite{Wang}]\label{blaschke2}
\textit{The Blaschke tensor is given in terms of the Moebius metric and the Moebius third fundamental form by }  \begin{equation}
(n-2)\psi(X,Y)=\Ric^{*}(X,Y)+III_{\beta}(X,Y)-\frac{n^{2}s^{*}+1}{2n}\langle X,Y\rangle^{*},
\end{equation} \textit{for all $X,Y \in \mathfrak{X}(M)$, where $\Ric^{*}$ and $s^{*}=\frac{1}{n(n-1)}\tr\, \Ric^{*}$ are the Ricci curvature and the scalar curvature of $(M^{n}, \langle \,,\,\rangle^{*}).$ In particular}, \begin{equation}\tr\, \psi=\frac{n^{2}s^{*}+1}{2n}=\frac{n}{2}\langle \mathcal{H}^{F}, \mathcal{H}^{F}\rangle. \end{equation}
\end{proposition}

\begin{proposition}[\cite{Wang}]\label{prop1}
\textit{The following equations hold:}
	
\begin{description}
	\item[(i)] \textit{The conformal Gauss equation}: \begin{align}\label{G}
	\langle R^{*}(X,Y)Z,W\rangle^{*}=&\langle \beta(X,W), \beta(Y,Z)\rangle - \langle \beta(X,Z), \beta(Y,W)\rangle\nonumber\\ 
	&+\psi(X,W)\langle Y,Z \rangle^{*}+\psi(Y,Z)\langle X,W\rangle^{*}\nonumber\\
	&-\psi(X,Z)\langle Y,W \rangle^{*}-\psi(Y,W)\langle X,Z \rangle^{*}
	\end{align} \textit{for all} $X,Y,Z,W \in \mathfrak{X}(M)$.
\item[(ii)] \textit{The Codazzi conformal equations}:
\begin{equation}\label{C1}
	({^{f}\nabla^{\perp}_{X}}\beta)(Y,Z)-({^{f}\nabla^{\perp}_{Y}}\beta)(X,Z)=\omega((X \wedge Y)Z)
\end{equation} \textit{and} \begin{equation}\label{C2}
(\nabla^{*}_{X}\psi)(Y,Z)-(\nabla^{*}_{Y}\psi)(X,Z)=\langle \omega(Y), \beta(X,Z)\rangle-\langle \omega(X), \beta(Y,Z)\rangle
\end{equation} \textit{for all} $ X,Y,Z \in \mathfrak{X}(M),$ \textit{where} $$ ({^{f}\nabla^{\perp}_{X}}\beta)(Y,Z)={^{f}\nabla^{\perp}_{X}}\beta(Y,Z)-\beta(\nabla^{*}_{X}Y,Z)-\beta(Y, \nabla^{*}_{X}Z),$$  $$(\nabla^{*}_{X}\psi)(Y,Z)=X(\psi(Y,Z))-\psi(\nabla^{*}_{X}Y, Z)-\psi(Y, \nabla^{*}_{X}Z)$$ \textit{and}
$$(X \wedge Y)Z=\langle Y,Z\rangle^{*} X-\langle X,Z\rangle^{*} Y.$$ 
\item[(iii)] \textit{The conformal Ricci equations}: \begin{equation}\label{R1}
	d\omega(X,Y)=\beta(Y, \hat{\psi}X)-\beta(X, \hat{\psi}Y)
\end{equation} \textit{and} \begin{equation}\label{R2}
\langle R^{\perp}(X,Y)\xi, \eta \rangle=\langle [B_{\xi}, B_{\eta}]X,Y\rangle^{*}
\end{equation} \textit{for all} $ X,Y \in \mathfrak{X}(M)$ \textit{and} $\xi, \eta \in \Gamma(N_{f}M)$, \textit{with} $ \hat{\psi} \in \Gamma(\End(TM))$ \textit{given by} \begin{equation*}
\langle \hat{\psi}X,Y\rangle^{*}=\psi(X,Y).
\end{equation*}
\end{description}	
\end{proposition}

\begin{theorem}[\cite{Wang}]
    \textit{Let $f,g:M^{n}\to \mathbb{R}^{m}$, $n \geq 2$, be immersions free of umbilical points. Then there exists a conformal transformation $\tau:\mathbb{R}^{m}\to \mathbb{R}^{m}$ such that $g=\tau \circ f$ if and only if $f$ and $g$ share the same Moebius metric and there exists a vector bundle isometry $\mathcal{T}:N_{f}M \to N_{g}M$ such that $$\mathcal{T} {^{f}\nabla^{\perp}}={^{g}\nabla^{\perp}}\mathcal{T} \quad \mbox{and} \quad  \mathcal{T} \circ \beta^{f}=\beta^{g}.$$}
\end{theorem}

\subsection{Submanifolds with flat normal bundle}

An isometric immersion $f:M^{n}\to \mathbb{Q}^{m}_{c}$ is said to have \emph{flat normal bundle} if the curvature tensor of its normal bundle vanishes identically. It is well-known  (see \cite{Re}) that, under this assumption, at each $x \in M^{n}$  the tangent space $T_{x}M$ decomposes orthogonally as 
   $T_{x}M=E_{\eta_1}(x)\oplus \ldots E_{\eta_s}(x)$,
   where $s=s(x)$ and $\eta_1, \ldots, \eta_s $ are the distinct principal normal vectors of $f$ at $x$. If $s(x)=k\in \mathbb{N}$ for all $x\in M^n$, then $f$ is said to be \emph{proper}. In this case, the Codazzi equations of $f$ are  equivalent to  \begin{align}
    \langle X_j,Y_j\rangle \nabla^{\perp}_{X_i}\eta_j&=\langle \nabla_{X_j}Y_j, X_i \rangle (\eta_j-\eta_i),\\
    \langle \nabla_{X_j}X_i, X_l \rangle (\eta_i-\eta_l)&=\langle \nabla_{X_i}X_j, X_l \rangle(\eta_j-\eta_l)\label{codazzi3},
\end{align}  if $X_i \in \Gamma(E_{\eta_i})$, $X_j,Y_j \in \Gamma(E_{\eta_j})$ and $X_l \in \Gamma(E_{\eta_l})$ for $1 \le i \ne j \ne l \ne i \le k$.    
 By the Gauss equation, the sectional curvature of $M^n$ at $x$ along the plane $\sigma$ spanned by  $X\in E_{\eta_i}$ and $Y\in E_{\eta_j}$ is 
 \begin{equation}\label{sec} K(\sigma)=c+\langle \eta_i(x), \eta_j(x)\rangle.
 \end{equation}

Let $f:M^{n}\to \mathbb{R}^{n+p}$ be a proper isometric immersion  with flat normal bundle and closed Moebius form. Let $\eta_1, \ldots, \eta_k$ be the principal normal vector fields of $f$ associated with the smooth distributions $E_{\eta_1}, \ldots, E_{\eta_k}$. 
 Given unit vector fields $X_i\in \Gamma(E_{\eta_i})$  and $X_j\in \Gamma(E_{\eta_j})$,
since $\langle \,,\,\rangle^{*}=\rho^{2}\langle \,,\,\rangle_{f}$ and $\beta= \rho(\alpha^{f}-\langle \,,\,\rangle_f\mathcal{H}^{f})$, then $\bar{X}_i=\rho^{-1}X_i$ and $\bar{X}_j=\rho^{-1}X_j$ are unit vector fields with respect to $\langle \,,\,\rangle^{*}$ such that  \begin{align*}
    \beta(\bar{X}_i, \bar{X}_j)&=\rho (\rho^{-2}\alpha^{f}(X_i,X_j)-\rho^{-2}\delta_{ij}\mathcal{H}^{f})\\
    &=\delta_{ij}\rho^{-1}(\eta_i-\mathcal{H}^{f}).
\end{align*} 
We call the normal vector fields $\bar{\eta}_i=\rho^{-1}(\eta_i-\mathcal{H}^{f})$, $1 \le i \le k$, the \emph{Moebius principal normal vector fields} of $f$.

\section{Examples}\label{ex}

 We start this section by computing the Moebius metric of an isometric immersion $f:M^{n}\to \mathbb{R}^{n+p}$ that is constructed as in Theorem \ref{cor:dupinpcn} in terms of an isometric   immersion $g:M^{p-\ell}_{\tilde{c}}\to \mathbb{Q}^{2p-\ell}_{\tilde{c}}$,   $0 \le \ell \le p-1$, of a Riemannian manifod with  constant sectional curvature $\tilde{c}$ that has flat normal bundle and vanishing index of relative nullity. Then we establish a condition that the mean curvature vector field $\mathcal{H}^{g}$ of  $g$ must satisfy in order that $f$ has constant Moebius curvature $c$. 

 Observe that, if  $g:M^{p-\ell}_{\tilde{c}}\to \mathbb{Q}^{2p-\ell}_{\tilde{c}}$, $0 \le \ell \le p-1$,  is an isometric immersion with flat normal bundle and vanishing index of relative nullity,  by the Gauss equation \eqref{sec} there exist $p-\ell$ distinct nowhere-vanishing pairwise orthogonal principal normal vector fields $\eta_1, \ldots, \eta_{p-\ell}$. 
 
 \begin{examples}\label{cilindro} 
$(i)$ \emph	{Let $g:M^{p-\ell}\to \mathbb{R}^{2p-\ell}$, $0 \le \ell \le p-1$, be an immersion with flat normal bundle, vanishing index of relative nullity and flat induced metric $ds^{2}$. By the \emph{cylinder in $\mathbb{R}^{n+p}$ over $g$} we mean the immersion given by 
 \begin{equation*}
 		f=g \times \Id:M^{p-\ell}\times \mathbb{R}^{n-p+\ell}\to \mathbb{R}^{2p-\ell}\times \mathbb{R}^{n-p+\ell}=\mathbb{R}^{n+p},
 	\end{equation*} 
where $\Id:\mathbb{R}^{n-p+\ell}\to \mathbb{R}^{n-p+\ell}$ is the identity map. Let $\eta_1, \ldots, \eta_{p-\ell}$ be the nowhere vanishing pairwise orthogonal principal normal vector fields of $g$ and let $X_1, \ldots, X_{p-\ell}$ be an orthonormal frame of $(M^{p-\ell}, ds^{2})$ such that $ \alpha^{g}(X_i,X_j)=\delta_{ij}\eta_i$ for $1 \le i, j \le p-\ell$.
 	Then 
 $$\rho^{2}=\frac{n}{n-1}(||\alpha^{f}||^{2}-n||\mathcal{H}^{f}||^{2})=\frac{n}{n-1}(||\alpha^{g}||^{2}-\frac{1}{n}||\alpha^{g}||^{2})=||\alpha^{g}||^{2}=(p-\ell)^{2}||\mathcal{H}^{g}||^{2}.$$
 	The induced metric by $f$ is $\langle \,,\,\rangle_{f}=ds^{2}+du_{p-\ell+1}^{2}+\ldots+du_{n}^{2}$,
 where $(u_{p-\ell+1}, \ldots, u_n)$ are the canonical coordinates on $\mathbb{R}^{n-p+\ell}$. Thus the Moebius metric determined by $f$ is  \begin{equation*}\label{moebiuseuclideo}\langle\,,\,\rangle^{*}=(p-\ell)^{2}||\mathcal{H}^{g}||^{2}(ds^{2}+du_{p-\ell+1}^{2}+\ldots+du_n^{2}).\end{equation*}}
\noindent $(ii)$ \emph{ Let $g:M^{p-\ell}\to \mathbb{S}^{2p-\ell}$, $0 \le \ell \le p-1$, be an isometric immersion with flat normal bundle, vanishing index of relative nullity and induced metric $ds^{2}$ of constant sectional curvature $1$.
	The \emph{generalized cone in $\mathbb{R}^{n+p}$ over $g$} is the immersion $f\colon M^{p-\ell}\times \mathbb{H}^{n-p+\ell}\to \mathbb{R}^{n+p}$ given by 
	\begin{equation*} f(x,z)=\Theta \circ (g,\Id)(x,z)=(z_1g(x), z_2, \ldots, z_{n-p+\ell}),\end{equation*} 
where $\mathbb{H}^{n-p+\ell}=\mathbb{R}_{+}\times \mathbb{R}^{n-p+\ell-1}$ is endowed with the hyperbolic metric $$dz^{2}=\frac{1}{z_1^{2}}(dz_1^{2}+\ldots+dz_{n-p+\ell}^{2}),$$ $\Id:\mathbb{H}^{n-p+\ell}\to \mathbb{H}^{n-p+\ell}$ is the identity map and $\Theta:\mathbb{S}^{2p-\ell}\times \mathbb{H}^{n-p+\ell}\to \mathbb{R}^{n+p}$ is the conformal diffeomorphism defined by $$ \Theta(y,z)=(z_1y, z_2, \ldots, z_{n-p+\ell}),$$ where its conformal factor is $z_1.$ The induced metric by  $f$ is \begin{equation*}
		\langle \,,\,\rangle_{f}=z_1^{2}(ds^{2}+dz^{2}).
	\end{equation*}}
	\emph{Let $\eta_1, \ldots, \eta_{p-\ell}$ be the nowhere vanishing pairwise orthogonal principal normal vector fields of $g$ and let $X_1, \ldots, X_{p-\ell}$ be an orthonormal frame of $(M^{p-\ell}, ds^{2})$ such that $ \alpha^{g}(X_i,X_j)=\delta_{ij}\eta_i$ for $1 \le i, j \le p-\ell$. The second fundamental form of $f$ is given by }
	\begin{equation*}
		\alpha^{f}\left(\frac{X_i}{z_1}, \frac{X_j}{z_1}\right)=\frac{1}{z_1}\alpha^{g}(X_i,X_j), \quad \alpha^{f}\left(\frac{d}{dz_1}, \frac{d}{dz_1}\right)=0=\alpha^{f}(\frac{d}{dz_1}, \frac{X_i}{z_1}).
	\end{equation*}
\emph{	Thus
	\begin{align*}
		\rho^{2}&=\frac{n}{n-1}(||\alpha^{f}||^{2}-n||\mathcal{H}^{f}||^{2})\\
		&=\frac{n}{n-1}(\sum_{i=1}^{p-\ell}||\alpha^{f}(\frac{X_i}{z_1},\frac{X_i}{z_1})||^{2}-n||\frac{1}{n}\sum_{i=1}^{p-\ell}\alpha^{f}(\frac{X_i}{z_1}, \frac{X_i}{z_1})||^{2})\\
		&=\frac{n}{n-1}(\frac{1}{z_1^{2}}\sum_{i=1}^{p-\ell}||\alpha^{g}(X_i,X_i)||^{2}-\frac{1}{n}||\frac{1}{z_1}\sum_{i=1}^{p-\ell}\alpha^{g}(X_i,X_i)||^{2})\\
		&=\frac{1}{z_1^{2}}\frac{n}{n-1}(\sum_{i=1}^{p-\ell}||\alpha^{g}(X_i,X_i)||^{2}-\frac{1}{n}\sum_{i=1}^{p-\ell}||\alpha^{g}(X_i,X_i)||^{2})\\
		&=\frac{1}{z_1^{2}}\sum_{i=1}^{p-\ell}||\alpha^{g}(X_i,X_i)||^{2}=\frac{1}{z_1^{2}}(p-\ell)^{2}||\mathcal{H}^{g}||^{2}.
	\end{align*}}\emph{Therefore, \begin{equation*}
		\rho=\frac{(p-\ell)||\mathcal{H}^{g}||}{z_1},
	\end{equation*} and  the Moebius metric determined by $f$ is \begin{equation*}\label{moebius-metric-esfera}
		\langle\,,\,\rangle^{*}=(p-\ell)^{2}||\mathcal{H}^{g}||^{2}(ds^{2}+dz^{2}).
	\end{equation*}}
\emph{	$(iii)$ 
Let $g:M^{p-\ell}\to \mathbb{H}^{2p-\ell}$, $0 \le \ell \le p-1$, be an isometric immersion with flat normal bundle, vanishing index of relative nullity  and whose induced metric $ds^{2}$ has constant sectional curvature $-1,$ where $$\mathbb{H}^{2p-\ell}=\{(z_1, \ldots, z_{2p-\ell}): z_{2p-\ell}>0\} $$ is endowed with the hyperbolic metric $$ dz^{2}=\frac{1}{z_{2p-\ell}^{2}}(dz_1^{2}+\ldots+dz_{2p-\ell}^{2}).$$}
\emph{ The \emph{rotational submanifold over} $g$ is the immersion $f:M^{p-\ell}\times \mathbb{S}^{n-p+\ell}\to \mathbb{R}^{n+p}$ defined by \begin{equation*}\label{rotational} f(x,y)=\Theta \circ (g, \Id)(x,y)=(g_1(x),\ldots,g_{2p-\ell-1}(x),g_{2p-\ell}(x)y),\end{equation*} where $\Id:\mathbb{S}^{n-p+\ell}\to \mathbb{S}^{n-p+\ell}$ is the identity map and $\Theta: \mathbb{H}^{2p-\ell}\times \mathbb{S}^{n-p+\ell}\to \mathbb{R}^{n+p}$ is the conformal diffeomorphism $$\Theta(z,y)=(z_1, \ldots, z_{2p-\ell-1}, z_{2p-\ell}y),$$ whose conformal factor is $z_{2p-\ell}$.}
	
	\emph{As a consequence, the induced metric by $f$ is $$ \langle \,,\,\rangle_{f}=g_{2p-\ell}^{2}(ds^{2}+dy^{2}),$$ where $dy^{2}$ is the canonical metric of $\mathbb{S}^{n-p+\ell}.$}
	
		\emph{Again, let $\eta_1, \ldots, \eta_{p-\ell}$ be the nowhere vanishing pairwise orthogonal principal normal vector fields of $g$ and let $X_1, \ldots, X_{p-\ell}$ be an orthonormal frame of $(M^{p-\ell}, ds^{2})$ such that $ \alpha^{g}(X_i,X_j)=\delta_{ij}\eta_i$ for $1 \le i, j \le p-\ell$.}
	
		\emph{The second fundamental forms of $f$ and $g$  are related by  $$\alpha^{f}=\Theta_{*}\alpha^{g}-\frac{1}{g_{2p-\ell}}\Theta_{*}((\grad\, z_{2p-\ell})\circ g)^{\perp}(ds^{2}+dy^{2}).$$ Here $\grad$ is computed with respect to $dz^{2}.$  In particular, \begin{align*}
		\alpha^{f}(\frac{X_i}{g_{2p-\ell}}, \frac{X_j}{g_{2p-\ell}})&=\begin{cases}
			\Theta_{*}\alpha^{g}(\frac{X_i}{g_{2p-\ell}}, \frac{X_i}{g_{2p-\ell}})-\frac{1}{g_{2p-\ell}^{3}}\Theta_{*}((\grad \, z_{2p-\ell}) \circ g)^{\perp}, \quad i=j,\\
			0, \quad i \ne j,
		\end{cases}\\
		\alpha^{f}(\frac{1}{g_{2p-\ell}}\frac{\partial}{\partial y_k}, \frac{1}{g_{2p-\ell}}\frac{\partial}{\partial y_r})&=\begin{cases}-\frac{1}{g_{2p-\ell}^{3}}\Theta_{*}((\grad\,z_{2p-\ell})\circ g)^{\perp}, \quad k=r\\
			0, \quad k \ne r
		\end{cases}.
	\end{align*}}
	
		\emph{We have \begin{align*}
		||\alpha^{f}||^{2}&=\sum_{i=1}^{p-\ell}||\alpha^{f}(\frac{X_i}{g_{2p-\ell}}, \frac{X_i}{g_{2p-\ell}})||^{2}+\sum_{i=p-\ell+1}^{n}||\alpha^{f}(\frac{1}{g_{2p-\ell}}\frac{\partial}{\partial y_i}, \frac{1}{g_{2p-\ell}}\frac{\partial}{\partial y_i})||^{2}\\
		&=\sum_{i=1}^{p-\ell}||\alpha^{g}(\frac{X_i}{g_{2p-\ell}}, \frac{X_i}{g_{2p-\ell}})-\frac{1}{g_{2p-\ell}^{3}}((\grad\, z_{2p-\ell})\circ g)^{\perp}||^{2}_{\Theta}+\sum_{i=p-\ell+1}^{n}||\frac{1}{g_{2p-\ell}^{3}}((\grad\, z_{2p-\ell})\circ g)^{\perp}||^{2}_{\Theta}\\
		&=\frac{1}{g_{2p-\ell}^{4}}\sum_{i=1}^{p-\ell}||\alpha^{g}(X_i, X_i)||^{2}_{\Theta}-\frac{2}{g_{2p-\ell}^{5}}\sum_{i=1}^{p-\ell}\langle \alpha^{g}(X_i, X_i), ((\grad \,z_{2p-\ell})\circ g)^{\perp}\rangle_{\Theta}\\&+ \frac{n}{g_{2p-\ell}^{6}}||((\grad \, z_{2p-\ell})\circ g)^{\perp}||^{2}_{\Theta}
	\end{align*} and  \begin{align*}
		||\mathcal{H}^{f}||^{2}&=||\frac{1}{n}\sum_{i=1}^{p-\ell}\alpha^{g}(\frac{X_i}{g_{2p-\ell}}, \frac{X_i}{g_{2p-\ell}})-\frac{1}{g_{2p-\ell}^{3}}((\grad\, z_{2p-\ell})\circ g)^{\perp}||^{2}_{\Theta}\\
		&=\frac{1}{n^{2}}\frac{1}{g_{2p-\ell}^{4}}\sum_{i=1}^{p-\ell}||\alpha^{g}(X_i, X_i)||^{2}_{\Theta}-\frac{2}{ng_{2p-\ell}^{5}}\sum_{i=1}^{p-\ell}\langle \alpha^{g}(X_i, X_i), ((\grad \, z_{2p-\ell})\circ g)^{\perp}\rangle_{\Theta}\\&+\frac{1}{g_{2p-\ell}^{6}}||((\grad \, z_{2p-\ell})\circ g)^{\perp}||^{2}_{\Theta}.
	\end{align*}
	Therefore \begin{align*}
		\rho^{2}=\frac{n}{n-1}(||\alpha^{f}||^{2}-n||\mathcal{H}^{f}||^{2})
		&=\frac{1}{g_{2p-\ell}^{4}}\frac{n}{n-1}(\sum_{i=1}^{p-\ell}||\alpha^{g}(X_i,X_i)||^{2}_{\Theta}-\frac{1}{n}\sum_{i=1}^{p-\ell}||\alpha^{g}(X_i,X_i)||^{2}_{\Theta})\\
		&=\frac{1}{g_{2p-\ell}^{4}}\sum_{i=1}^{p-\ell}||\alpha^{g}(X_i,X_i)||^{2}_{\Theta}\\
		&=\frac{1}{g_{2p-l}^{2}}\sum_{i=1}^{p-\ell}||\alpha^{g}(X_i,X_i)||^{2}_{dz^{2}}=\frac{(p-\ell)^{2}||\mathcal{H}^{g}||^{2}_{dz^{2}}}{g_{2p-\ell}^{2}},
	\end{align*}that is, \begin{equation*}\rho=\frac{(p-\ell)||\mathcal{H}^{g}||_{dz^{2}}}{g_{2p-\ell}}.\end{equation*}
	We conclude that the Moebius metric determined by $f$ is given by  $$ \langle\,,\,\rangle^{*}=(p-\ell)^{2}||\mathcal{H}^{g}||^{2}_{dz^{2}}(ds^{2}+dy^{2}).$$}
\end{examples} 

\begin{lemma}\label{moebius-constante}
	\textit{Let $g:M^{p-\ell}\to \mathbb{Q}^{2p-\ell}_{\tilde{c}}$, $0 \le \ell \le p-2$, be an isometric immersion with flat normal bundle and vanishing index of relative nullity.  Then the immersion $f=\Theta \circ (g,\Id)\colon M^{p-\ell}\times \mathbb{Q}^{n-p+\ell}_{-\tilde{c}} \to \mathbb{R}^{n+p}$, where $\Theta=\Id:\mathbb{R}^{n+p}\to \mathbb{R}^{n+p}$ if $\tilde{c}=0$ or $\Theta$ is as in Examples \ref{cilindro}-$(ii)$ and $(iii)$ if  $\tilde{c} \ne 0$, has constant Moebius curvature $c$ if and only if the induced metric $ds^{2}$ by $g$ has constant sectional curvature $\tilde{c}$ and 
	\begin{equation}\label{condicoes}
		\Hess\, (1/||\mathcal{H}^{g}||)+\tilde{c}(1/||\mathcal{H}^{g}||)ds^{2}=0 \quad \mbox{and} \quad ||\grad\, (1/||\mathcal{H}^{g}||)||^{2}+\tilde{c}(1/||\mathcal{H}^{g}||^{2})=-(p-\ell)^{2}c,
	\end{equation} 
	for $\grad,$ $\Hess$ and $||\cdot|| $  computed with respect to $ds^{2}.$}
\end{lemma}

We make use of the following well-known fact.

\begin{lemma}\label{constante}
Let $(M^{p}, g_1)$ and $(N^{n-p},g_2)$ be Riemannian manifolds with $n-2 \geq p \geq 1$. The warped metric $g_1+\mu^{2} g_2$ on $M^{p}\times N^{n-p}$, for $\mu \in C^{\infty}(M^{p})$, has constant sectional curvature $c$ if and only if \begin{itemize}
\item[(i)] $g_1$ has constant sectional curvature $c$ (for $p \geq 2$);
\item[(ii)] $\Hess \, \mu+c\mu g_1=0$; 
\item[(iii)] $g_2$ has constant sectional curvature $||\grad \, \mu||^{2}+c \mu^{2}$, 
\end{itemize}
where $\Hess$ and $\grad $ are computed with respect to $g_1$.
\end{lemma}

\begin{proof}
    We start by proving the converse assertion.  Assume  $(M^{p-\ell}, ds^{2})$ has constant curvature $\tilde{c}$. By Examples \ref{cilindro},  the Moebius metric determined by $f$ is given by \begin{equation}\label{moebiusgeneral}
		\langle\,,\,\rangle^{*}=(p-\ell)^{2}||\mathcal{H}^{g}||^{2}(ds^{2}+\sigma_{-\tilde{c}}),
	\end{equation} here $\sigma_{-\tilde{c}}$ stands for the metric of $\mathbb{Q}^{n-p+\ell}_{-\tilde{c}}.$
	Write $$\langle\,,\,\rangle^{*}=g_1+\mu^{2}g_2,$$ where $\mu=(p-\ell)||\mathcal{H}^{g}||$, $g_1=\mu^{2}ds^{2}$ and $g_2=\sigma_{-\tilde{c}}.$ It follows from Lemma \ref{constante} that $g_1$ has constant sectional curvature $c$ (for $p \geq 2$), $\Hess \, \mu+c\mu g_1=0$ and $g_2$ has constant sectional curvature $||\grad \, \mu||^{2}+c \mu^{2}$, 
where $\Hess$ and $\grad $ are computed with respect to $g_1$.

 Conversely, if $f$ has constant Moebius curvature $c$,  by Example  \ref{cilindro} the Moebius metric determined by $f$ is given by \eqref{moebiusgeneral}.  Thus \eqref{condicoes} holds by Lemma \ref{constante}.
\end{proof}

The next corollary gives the solutions of  \ref{condicoes} on open subsets $U\subset\mathbb{Q}^{p-\ell}_{\tilde{c}}$.

\begin{corollary}\label{meancurvature}
	\textit{The solutions of \eqref{condicoes} on an open subset $U$ of $\mathbb{Q}^{p-\ell}_{\tilde{c}}$ are \begin{equation}\label{mean-curvature} ||\mathcal{H}^{g}||(x)=\begin{cases}\frac{1}{\langle v,x\rangle+a}, \quad \tilde{c}=0  \,\,\, \mbox{and}\,\,\, c \le 0  \vspace{8pt}\\
			\frac{1}{\langle v,x\rangle}, \quad \tilde{c}=1\,\, \mbox{and} \,\, c <0 \,\, \mbox{or}\,\, \Tilde{c}=-1 \,\, \mbox{and}\,\, c \in \mathbb{R}
		\end{cases},
	\end{equation} where $v \in \mathbb{E}^{p-\ell+1}:=\begin{cases}
 \mathbb{R}^{p-\ell}, \qquad \mbox{if} \, \, \, \Tilde{c}=0\\
	    \mathbb{R}^{p-\ell+1}, \quad \mbox{if}\,\, \, \Tilde{c}=1\\
     \mathbb{L}^{p-\ell+1}, \quad \mbox{if}\,\,\, \Tilde{c}=-1
	\end{cases}$ is such that $||v||^{2}=-(p-\ell)^{2}c$ and $a \in \mathbb{R}_{+}$. }
\end{corollary}

\begin{proof}
	Consider $(U,ds^{2})$ as an open subset of $\mathbb{R}^{p-\ell}$. Then 
	\begin{equation*}
		\frac{1}{||\mathcal{H}^{g}||}(x)=\langle v,x\rangle +a \quad \forall \, x \in U,
	\end{equation*} where $a \in \mathbb{R}_{+}$ and $v \in \mathbb{R}^{p-\ell}$ is a constant vector such that $||v||^{2}=-(p-\ell)^{2}c$, is the solution of \eqref{condicoes} for $\tilde{c}=0$.

	Now, consider $(U, ds^{2})$ as an open subset of $\mathbb{S}^{p-\ell}\subset \mathbb{R}^{p-\ell+1}$ and let  $g:\mathbb{R}^{p-\ell+1} \to \mathbb{R}$  be the linear functional defined by 
	$$g(x)=\langle x,v\rangle ,\quad \mbox{where}\,\,v \in \mathbb{R}^{p-\ell+1} \,\, \mbox{is such that}\,\, ||v||^{2}=-(p-\ell)^{2}c.
	$$
	Then the gradient and  Hessian of $g$ and those of  $h=g \circ i\colon U \subset \mathbb{S}^{p-\ell}\to \mathbb{R},$ where $i:\mathbb{S}^{p-\ell}\to \mathbb{R}^{p-\ell+1}$ is an umbilical inclusion, are related by $$ i_{*} \grad\,  h=(\grad \, g)^{T}$$ and $$\Hess\, h(X,Y)=\Hess\, g(i_{*}X, i_{*}Y)+\langle \grad \, g, \alpha^{i}(X,Y)\rangle$$ for any $x \in U$ and $X,Y \in T_{x}\mathbb{S}^{p-\ell}$ (see Proposition 1.2 of \cite{DT}). Since $\grad\, g=v $ e $\alpha^{i}(X,Y)=-\langle X,Y\rangle x$, we obtain $$i_{*}\grad \, h=v^{T} $$ and $$ \Hess \, h(X,Y)=-\langle v,x\rangle \langle X,Y \rangle=-h(x)\langle X,Y\rangle.$$
	Noticing \begin{align*}
		||\grad \, h||^{2}=||v^{T}||^{2}=||v-v^{\perp}||^{2}=||v||^{2}-||v^{\perp}||^{2}&=-(p-\ell)^{2}c-\langle v,x\rangle^{2}||x||^{2}\\
		&=-(p-\ell)^{2}c-\langle v,x\rangle^{2},
	\end{align*} then
	 $$ ||\grad \, h||^{2}+(p-\ell)^{2}c=-h^{2}.$$ 
	Therefore, $1/||\mathcal{H}^{g}||:U \subset \mathbb{S}^{p-\ell}\to \mathbb{R}$ defined by \begin{equation*}
		(1/||\mathcal{H}^{g}||)(x)=\langle x,v\rangle, \quad \forall x \in U,
	\end{equation*} with $||v||^{2}=-(p-\ell)^{2}c$, is a solution of \eqref{condicoes} for $\tilde{c}=1.$

	Similarly, consider $(U,ds^{2})$ as an open subset of $\mathbb{H}^{p-\ell}\subset \mathbb{L}^{p-\ell+1},$ where $\mathbb{H}^{p-\ell}$ is endowed with the hyperboloid model. Let $g:\mathbb{L}^{p-\ell+1}\to \mathbb{R}$ be the linear functional defined by 
	$$g(x)=\langle x,v\rangle, \,\, \mbox{where}\,\,v \in \mathbb{L}^{p-\ell+1} \,\, \mbox{is such that}\,\, ||v||^{2}=-(p-\ell)^{2}c,$$ 
	and let $h:=g \circ i:U \subset \mathbb{H}^{p-\ell}\to \mathbb{R}$, where $i:\mathbb{H}^{p-\ell}\to \mathbb{L}^{p-\ell+1}$ is an umbilical inclusion. 
	
	The Hessian and gradient of $g$ and $h$ are related by $$ i_{*}\grad \, h=(\grad \, g)^{T}$$ and $$ \Hess\, h (X,Y)=\Hess \, g(i_{*}X, i_{*}Y)+\langle \grad \, g, \alpha^{i}(X,Y)\rangle,$$ for any $x \in U$ and $X,Y \in T_{x}\mathbb{H}^{p-\ell}.$ Since $\grad \, g=v$ and $\alpha^{i}(X,Y)=\langle X,Y\rangle x$, we have $$ i_{*}\grad \, h=v^{T}$$ and $$\Hess \, h(X,Y)=\langle v,x \rangle \langle X,Y\rangle=h(x)\langle X,Y\rangle.$$
	Thus \begin{align*}
		||\grad \, h ||^{2}&=||v^{T}||^{2}=||v-v^{\perp}||^{2}=||v||^{2}-||v^{\perp}||^{2}\\
		&=-(p-\ell)^{2}c-\langle v^{\perp}, x\rangle^{2}||x||^{2}\\
		&=-(p-\ell)^{2}c+\langle v,x\rangle^{2}\\
		&=-(p-\ell)^{2}c+(h(x))^{2}.
	\end{align*}
	Therefore, the function $1/||\mathcal{H}^{g}||:U \subset \mathbb{H}^{p-\ell}\to \mathbb{R}$ given by \begin{equation*}
		(1/||\mathcal{H}^{g}||)(x)=\langle v,x \rangle, \quad \mbox{where} \,\, v \in \mathbb{L}^{p-\ell+1} \,\, \mbox{is such that}\,\, ||v||^{2}=-(p-\ell)^{2}c , 
	\end{equation*} is a solution of \eqref{condicoes} for $\tilde{c}=-1.$
\end{proof}

\begin{observation} \emph{Particular cases of isometric immersions $g:M^{p-\ell}_{\tilde{c}}\to \mathbb{Q}^{2p-\ell}_{\Tilde{c}}$, $0 \le \ell \le p-2$, satisfying the conditions of Lemma \ref{moebius-constante} are those for which $||\mathcal{H}^{g}||$ is constant and $\tilde{c}=0$.  From a global point of view, it was proved in \cite{DT-constante-curvature-sub} that if $g\colon\mathbb{R}^{n}\to \mathbb{R}^{m}$ is an isometric immersion with flat normal bundle, constant index of relative nullity and mean curvature vector field with constant length, then $g(\mathbb{R}^{n})$ is a Riemannian product of curves with constant first curvature functions. The case $n=2$ was obtained previously in  \cite{enomoto} without assuming the index of relative nullity to be constant. In particular, if $g\colon M^{2}\to \mathbb{R}^{4}$ is a compact surface   with flat induced metric and flat normal bundle whose mean curvature vector field has constant length then $g(M^{2})$ is a Riemannian product of two circles $\mathbb{S}^{1}\times \mathbb{S}^{1}.$}
\end{observation}

\begin{example}
   \emph{ Let $\gamma_i\colon I_i \to \mathbb{R}^{2}$ be a smooth curve with curvature $\kappa_i$, $1 \leq i \leq 2.$ Consider the surface $ g=\gamma_1 \times \gamma_2\colon I_1 \times I_2 \to \mathbb{R}^{2}\times \mathbb{R}^{2}=\mathbb{R}^{4}$, for which
    $||\mathcal{H}^{g}||^{2} =\frac{1}{4}(\kappa_1^{2}+\kappa_2^{2})$.
In particular, \begin{equation*}
    \Hess\, \| \mathcal{H}^{g}\|^{2}(\partial_1, \partial _2)=\partial_1(\partial_2( \kappa_1^{2}+\kappa_2^{2}))-(\nabla_{\partial_1}\partial_2)(\kappa_1^{2}+\kappa_2^{2})=0,
\end{equation*} where $\partial_1$ and $\partial_2$ are the coordinates vector fields of $I_1$ and $I_2$, respectively, and $\nabla$ is the Levi-Civita connection of the product metric $ds_1^{2}+ds_2^{2}.$}
\emph{ Assume $ \|\mathcal{H}^{g}\|^{-1}(x)=\langle v,x\rangle+a$ for all $x \in I_1 \times I_2$, where $||v||=2\sqrt{-c}$ for $c\le 0$ and $a \in \mathbb{R}_{+}$. Then \begin{align*}
   0= \Hess \, ||\mathcal{H}^{g}||^{-1}(\partial_1, \partial_2)&=-\frac{1}{2}||\mathcal{H}^{g}||^{-3}\Hess\, ||\mathcal{H}^{g}||^{2}(\partial_1, \partial_2)+\frac{3}{4}||\mathcal{H}^{g}||^{-5}(\partial_1||\mathcal{H}^{g}||^{2})\partial_2||\mathcal{H}^{g}||^{2}\\
    &=\frac{3}{4}||\mathcal{H}^{g}||^{-5}(\partial_1||\mathcal{H}^{g}||^{2})\partial_2||\mathcal{H}^{g}||^{2},
\end{align*} hence $$(\partial_1 \kappa_1)(\partial_2 \kappa_2)=0.$$}
 \emph{ If, say, $k_1$ is not constant, then $\kappa_2\equiv r$ for some $r>0$ and, from $\frac{1}{||\mathcal{H}^{g}||}(x_1)=2\sqrt{-c}x_1$ for $c<0$, we obtain $$\kappa_1(x_1)=\sqrt{-c^{-1}x_1^{-2}-r^{2}}, \quad \mbox{for}\,\,\, 0<|x_1| < \frac{1}{r\sqrt{-c}}.$$ Therefore, $g=\gamma_1 \times \gamma_2$, where $\gamma_1$ is a curve with curvature $\kappa_1(x_1)=\sqrt{-c^{-1}x_1^{-2}-r^{2}}$ and $\gamma_2(I_2)\subset \mathbb{S}^{1}$, hence $g \times \Id:I_1 \times I_2 \times \mathbb{R}^{n-2}\to \mathbb{R}^{n+2}$ is a cylinder  with constant negative Moebius curvature $c.$}
\end{example}

\section{Moebius invariants of conformally flat submanifolds with closed Moebius form}\label{moebius-inva-closed}

Let $f:M^{n}\to \mathbb{R}^{n+p}$, $n-3 \geq p \geq 1,$ be a proper isometric immersion with flat normal bundle  of a conformally flat manifold. Under these assumptions, Theorem \ref{theorem-moore} assures that $f$ has a principal normal vector field $\eta$ with multiplicity $n-p +\ell$, for some $0 \le \ell \le p-1$, whereas Theorem \ref{theorem-DOV} implies that all the other $p-\ell$ principal normal vector fields of $f$ are  simple.

  Assume that the Moebius form of $f$ is closed. Let $\bar{\eta}_i=\rho^{-1}(\eta_i-\mathcal{H}^{f})$, $1 \le i \le k,$ be  the associated Moebius principal normal vector fields of $f$. Since the Moebius form $\omega^{f}$ is closed, the conformal Ricci equation implies that there exists an orthonormal frame $X_1, \ldots, X_n$ with respect to $\langle \,,\,\rangle^{*}$ that diagonalizes $\beta$ and $\psi$ simultaneously and such that 
  $$ \beta(X_i,X_i)=\Bar{\eta}:=\rho^{-1}(\eta-\mathcal{H}^{f})$$ for any $p- \ell +1 \le i \le n$.  Furthermore, all the remaining Moebius principal normal vector fields $\beta(X_j,X_j)$, $1 \le j \le p-\ell$, are  simple.  
 Now consider the smooth distributions \begin{equation}
    \Delta=\spa\{X_i: p-\ell+1 \le i \le n\} \quad \mbox{and} \quad \Delta^{\perp}=\spa\{X_i: 1 \le i \le p-\ell\}.
\end{equation}
Since $\Delta=E_{\eta}$ and $\dim \Delta=n-p +\ell \geq 2,$ then $\Delta$ is umbilical with respect to $\langle \,,\,\rangle_{f}$  by Proposition \ref{dupin},  hence also umbilical with respect to $\langle \,,\,\rangle^{*}$, for $\langle \,,\,\rangle^{*}$ and $\langle \,,\,\rangle_{f}$ are conformal metrics.

\begin{proposition}\label{totally-umbilical}
    If the distribution $\Delta^{\perp}$ is totally  geodesic with respect to $\langle\,,\,\rangle^{*}$, then $\Delta^{\perp}$ is spherical with respect to $\langle \,,\rangle_{f}$ with mean curvature vector field $(\grad_{f}\log \rho)_{\Delta}.$
\end{proposition}

\begin{proof}
Using the relation $$\nabla^{*}_{X_a}X_b={^{f}\nabla_{X_a}}X_b+\frac{1}{\rho}(X_a(\rho)X_b+X_b(\rho)X_a-\langle X_a,X_b\rangle_{f}\grad_{f}\rho)$$ and the fact that $\Delta^{\perp}$ is totally geodesic with respect to $\langle\,,\,\rangle^{*}$, we obtain \begin{equation*}\label{total} \langle {^{f}\nabla_{X_a}X_b}, X_i\rangle_{f}=\langle X_a,X_b\rangle_{f}\langle \frac{\grad_{f}\rho}{\rho}, X_i\rangle_{f}\end{equation*} for all $i \geq p-\ell+1$ and $1 \le a,b \le p-\ell.$ Thus $\Delta^{\perp}$ is umbilical with respect to the metric induced by $f$, with $(\grad_{f} \log \rho)_{\Delta}$ as its mean curvature vector field. 

 In order to prove that $\Delta^{\perp}$ is  spherical with respect to $\langle \,,\,\rangle_{f}$, we must show that
  \begin{equation}\label{spherical}
			\langle {^{f}\nabla_{X_i}}(\grad_{f} \log \rho )_{\Delta}, X_j \rangle_{f}=0, \quad 1\leq i\leq p-\ell, \quad  p-\ell+ 1 \le j \le n.
   \end{equation}
   Since $X_1, \ldots, X_n$  diagonalizes $\beta$ and $\psi$ simultaneously, then $\Hess^{*}\rho(X_i,X_j)=0$ for all $ 1 \leq i\neq j \leq n$. Using that
 \begin{equation*}
			\nabla^{*}_{X_i}\grad^{*}\rho= {^{f}\nabla_{X_i}}\grad^{*}\rho+\frac{1}{\rho}((\grad^{*}\rho)(\rho)X_i+ X_i(\rho)\grad^{*}\rho-\langle X_i, \grad^{*}\rho \rangle_{f}\grad_{f}\rho),
		\end{equation*} and that $
		\grad^{*}\rho=\dfrac{\grad_{f}\rho}{\rho^{2}},
		$ it follows that 
  \begin{align}
			0&=  \Hess^{*}\rho(X_i,X_j)\nonumber\\&=\langle \nabla^{*}_{X_i}\grad^{*}\rho, X_j \rangle^{*}\nonumber\\
   &=\rho^2 
			\langle {^{f}\nabla_{X_i}}\grad^{*}\rho, X_j \rangle_{f}\nonumber\\
			&=\langle X_i(\rho^{-1})\grad_{f}\log \rho+\rho^{-1}{^{f}\nabla_{X_i}}\grad_{f}\log \rho, X_j \rangle_{f}\nonumber,
		\end{align} hence \begin{equation}\label{esfera1}
			\langle ^{f}\nabla_{X_i}\grad_{f}\log \rho, X_j \rangle_{f}=X_i(\log \rho)X_j(\log \rho).\end{equation}
			 On the other hand, since $\nabla^{*}_{X_i}X_j={^{f}\nabla_{X_i}}X_j+\frac{1}{\rho}(X_i(\rho)X_j+X_j(\rho)X_i) \in \Gamma(\Delta)$, then 
   \begin{align}
			\langle ^{f}\nabla_{X_i}(\grad_{f}\log \rho)_{\Delta^{\perp}},X_j \rangle_{f}&=-\langle (\grad_{f}\log \rho)_{\Delta^{\perp}}, {^{f}\nabla_{X_i}}X_j \rangle_{f}\nonumber\\
			&=X_i(\log \rho)X_j(\log \rho), \label{esfera2}
		\end{align} 
  and hence \eqref{spherical} follows from \eqref{esfera1} and \eqref{esfera2}. 
\end{proof}

\subsection{Submanifolds with exactly two distinct principal normals}

The next result classifies isometric immersions with flat normal bundle and arbitrary codimension $f:M^{n}\to \mathbb{R}^{n+p}$, $n \geq 3$, that have closed Moebius form and a principal normal vector field $\eta$ of multiplicity $n-1$. It  generalizes Theorem 5.3 in \cite{deform} and Corollary 1.2 of \cite{hyper-two-curvatures-closed-moebius-form}, which classify umbilic-free conformally flat hypersurfaces with closed Moebius form.

\begin{proposition}\label{prop.reduction}
Let $f:M^{n}\to \mathbb{R}^{n+p}$, $n \geq 3$ and $p \geq 1$, be an isometric immersion with flat normal bundle that has a principal normal vector field $\eta$ of multiplicity $n-1$. If $\omega^{f}$ is closed, then $f(M)$ is the image by a conformal transformation of  $\mathbb{R}^{n+p}$ of an open subset of a cylinder, a generalized cone or a rotational submanifold over a curve $\gamma:I\to \mathbb{Q}^{p+1}_{\tilde{c}}$, where $\tilde{c}=0,1$ or $-1$, respectively.
\end{proposition}

\begin{proof}
    Let $\bar{\eta}=\beta(X_i,X_i)$ for $2 \le i \le n$, and consider the smooth distributions $$\Delta=\spa\{X_i: 2 \le i \le n\} \quad \mbox{and} \quad \Delta^{\perp}=\{X_1\}.$$
    Since $\tr \, \beta =0$ and $||\beta||^{2}_{*}=\frac{n-1}{n}$, we have $$\beta(X_1,X_1)+(n-1)\bar{\eta}=0 \quad \mbox{and} \quad ||\beta(X_1,X_1)||^{2}+(n-1)||\bar{\eta}||^{2}=\frac{n-1}{n},$$ 
    hence 
    $$||\beta(X_1,X_1)||=\frac{n-1}{n}\quad \mbox{and} \quad ||\bar{\eta}||=\frac{1}{n}.$$ 
 Therefore, there is a  unit normal vector field $\xi_1 \in \Gamma(N_{f}M)$ such that \begin{equation}\label{moebiusnormal1}
	    \beta(X_1,X_1)=\frac{n-1}{n}\xi_1 \quad \mbox{and} \quad \bar{\eta}=-\frac{1}{n}\xi_1.
	\end{equation} 
The conformal Codazzi equation \begin{equation*}
		    ({^{f}\nabla^{\perp}_{X_j}}\beta)(X_i,X_i)-({^{f}\nabla^{\perp}_{X_i}}\beta)(X_j,X_i)=\omega((X_j \wedge X_i)X_i)
		\end{equation*} for $2 \le i \ne j \le n$ yields $\omega(X_j)={^{f}\nabla^{\perp}_{X_j}}\bar{\eta}$.   On the other hand, from \begin{equation*}
		    ({^{f}\nabla^{\perp}_{X_j}}\beta)(X_1,X_1)-({^{f}\nabla^{\perp}_{X_1}}\beta)(X_j,X_1)=\omega((X_j \wedge X_1)X_1),
		\end{equation*} for $2 \le j \le n$, we obtain \begin{equation*}
		    \omega(X_j)={^{f}\nabla^{\perp}_{X_j}}\beta(X_1,X_1)+\langle \nabla^{*}_{X_1}X_1, X_j \rangle^{*}(\bar{\eta}-\beta(X_1,X_1)).
		\end{equation*} Since $\beta(X_1,X_1)=-(n-1)\bar{\eta}$, then \begin{equation}\label{l1}
		{^{f}\nabla^{\perp}_{X_j}\bar{\eta}}=\langle \nabla^{*}_{X_1}X_1, X_j \rangle^{*}\bar{\eta}.
		\end{equation}
  Taking the inner product of the preceding equation with $\bar{\eta}$ gives $
		    \langle \nabla^{*}_{X_1}X_1, X_j \rangle^{*}=0$  for $2 \le j \le n$. Therefore, $\Delta^{\perp}$ is totally geodesic with respect to the Moebius metric, hence spherical with respect to $\langle\,,\,\rangle_{f}$ with $(\grad_{f}\log \rho)_{\Delta}$ as its mean curvature vector field by Proposition \eqref{totally-umbilical}. Taking into account that  $\Delta=E_{\eta}$ and that $\dim \Delta~\geq 2$, the statement follows from Theorem \ref{cor:dupinpcn}.
\end{proof}

The next lemma will also be used in the proof of Theorem \ref{theorem-main}. 

 \begin{lemma}\label{moebius-closed}
		Let $f:M^{n}\rightarrow \mathbb{R}^{n+p}$ be an isometric immersion with constant Moebius curvature and flat normal bundle. Then the Moebius form of $f$ is closed.
	\end{lemma}
 \begin{proof}
     The conformal Ricci equation $$\langle R^{\perp}(X,Y)\xi, \eta \rangle=\langle [B_{\xi}, B_{\eta}]X,Y\rangle^{*}$$ for all $X,Y \in \mathfrak{X}(M)$ and $\xi, \eta \in \Gamma(N_{f}M)$, shows that there exists an orthonormal frame $\mathcal{B}:=\{X_1, \ldots, X_n\}$ with respect to the Moebius metric such that $$\beta(X_i,X_j)=0, \quad i \ne j.$$

     Let $c \in \mathbb{R}$ be the common value of the sectional curvatures of $(M^{n}, \langle \,,\,\rangle^{*})$. Then the Ricci tensor of $(M, \langle \,,\,\rangle)$ satisfies $$\Ric^{*}(X,Y)=c(n-1)\langle X,Y\rangle^{*}$$ for all $X,Y \in \mathfrak{X}(M)$. Hence $\mathcal{B}$ diagonalizes $\Ric^{*}$. 
     Using the relation \begin{equation*}\label{relation-blaschke}
			(n-2)\psi(X,Y)=\Ric^{*}(X,Y)+III_{\beta}(X,Y)-\frac{n^{2}s^{*}+1}{2n}\langle X,Y \rangle^{*}, 
		\end{equation*} where $s^{*}$ denotes the scalar curvature of $(M, \langle \,,\,\rangle^{*}),$ we conclude that the Blaschke tensor $\psi$ is also diagonalizable by $\mathcal{B}$. Therefore $\omega$ is closed  by \eqref{R1}.
 \end{proof}

The next result contains as a  particular case the classification of umbilic-free hypersurfaces with constant Moebius curvature in Theorem $1.1$ of
 \cite{classif-moebius-hyper} and  Theorem $4.9$ of \cite{deform}.

\begin{lemma}\label{ode}
	The Moebius metric $\langle \,,\,\rangle^{*}=\kappa^{2}(ds^{2}+\sigma_{-\Tilde{c}})$ of the isometric immersion $f:I \times \mathbb{Q}_{-\Tilde{c}}^{n-1}\to \mathbb{R}^{n+p}$ defined by $$f=\Theta \circ (\gamma, \Id),$$ where $\sigma_{-\Tilde{c}}$ denotes the canonical metric of $\mathbb{Q}^{n-1}_{-\Tilde{c}}$, $\gamma:I \to \mathbb{Q}^{p+1}_{\Tilde{c}}$ and $\Theta=\Id: \mathbb{R}^{n+p}\to \mathbb{R}^{n+p}$ if $\Tilde{c}=0$ or $\Theta$ is as in Examples \ref{cilindro}- $(ii)$ and $(iii)$  if $\Tilde{c} \ne 0$, has constant Moebius curvature $c$ if and only if the first curvature function $\kappa(s)$ is given respectively by $$\kappa(s)=\begin{cases}
			\frac{1}{r}, \quad \,\quad c=0 \,\, \mbox{and} \,\, s\in \mathbb{R}\\
			\frac{1}{\sqrt{-c}s}, \quad c<0 \,\, \mbox{and}\,\, s>0,
		\end{cases}$$
		$$ \kappa(s)=\frac{1}{\sqrt{-c}\sin s}, \,\,\,s \in (0, \pi), \,\,\,\,c<0,$$
		and
		$$\kappa(s)=\begin{cases}
			\frac{1}{\sqrt{c}\cosh{s}}, \,\,\quad c>0 \,\,\mbox{and}\,\, s \in \mathbb{R}\\
			\frac{1}{\sqrt{-c}\sinh{s}}, \quad c<0\,\, \mbox{and} \,\, s>0\\
			e^{s}, \qquad \qquad c=0\,\, \mbox{and}\,\, s \in \mathbb{R}.
		\end{cases}$$
	\end{lemma}

   \begin{theorem}\label{theorem-main-2}  Let $f:M^{n}\to \mathbb{R}^{n+p}$,  $n \geq 3$, be an isometric immersion with constant Moebius curvature $c$ and flat normal bundle with exactly two distinct principal normal vector fields. Then $f(M)$ is the image by a conformal transformation of $\mathbb{R}^{n+p}$ of an open subset of a submanifold given in Example \ref{cilindro}, with the first curvature function $\kappa(s)$ of the curve $\gamma\colon I\to \mathbb{Q}_{\tilde{c}}^{p+1}$ given by Lemma \ref{ode},  with $\tilde{c}=0, 1$ or $-1$, respectively.
\end{theorem} 

\begin{proof}     Since $f$ has constant Moebius curvature and flat normal bundle, the Moebius form $\omega^f$ is closed by Lemma \ref{moebius-closed}. By the conformal Ricci equation,  there exists an orthonormal frame $X_1, \ldots X_n$ with respect to the metric $\langle\,,\,\rangle^{*}$  that diagonalizes $\beta$ and $\psi$ simultaneously, with $\beta(X_i,X_i)=\bar{\eta}$  for all $ 2 \le i  \le n.$  
     By Proposition \ref{prop.reduction}, $f(M)$ is the image by a conformal transformation of $\mathbb{R}^{n+p}$ of an open subset of a cylinder, a generalized cone or a rotational submanifold over a curve $\gamma:I \to \mathbb{Q}^{p+1}_{\Tilde{c}}$, where $\Tilde{c}=0,1$ or $-1$, respectively. Finally, the assumption that the Moebius curvature of $f$ is constant
     implies that the first curvature function $\kappa(s)$ of  $\gamma$ is given as in Lemma \ref{ode}. \end{proof}

\subsection{Submanifolds with at least three principal normal vector fields}

 Let $f:M^{n}\to \mathbb{R}^{n+p}$, $n-3\geq p \geq 2$, be a proper conformally flat isometric immersion with flat normal bundle whose Moebius form is closed and let $\eta$ be a   principal normal vector field that has multiplicity $n-p+\ell$,  $0 \le \ell \le p-2$, whose existence is guaranteed by Theorem \eqref{theorem-moore}. Let $\bar \eta$ be the associated Moebius principal normal vector field. We start by finding a suitable orthogonal frame of $N_fM$. 

\begin{lemma}\label{lema-ortogonal}
The subset $\{\beta(X_i,X_i)-\Bar{\eta}: 1 \le i \le p-\ell\}$ of $N_{f}M$ is orthogonal and $\Delta^{\perp}:=E_\eta^\perp$ is integrable.
\end{lemma}

\begin{proof}
    The conformal Gauss equation of $f$ reduces to \begin{equation}\label{eq.gauss}
		K^{*}(X_k,X_r)=\langle R^{*}(X_k,X_r)X_r,X_k\rangle^{*}=\langle \beta(X_k,X_k), \beta(X_r,X_r)\rangle+\psi(X_k,X_k)+\psi(X_r,X_r),
	\end{equation} for all $1\leq k \ne r \leq n$.  Since $(M, \langle \,,\,\rangle^{*})$ is conformally flat,   it follows from Kulkarni's formula (see \cite{Kulkarni}) that $$K^{*}(X_i,X_j)+K^{*}(X_k,X_r)=K^{*}(X_i,X_k)+K^{*}(X_j,X_r)$$ for $k \ne r \geq p-\ell+1$ and $1 \le i \ne j \le p-\ell$. Thus 
	\begin{equation}\label{on}
		\langle \beta(X_i,X_i)-\bar{\eta}, \beta(X_j,X_j)-\bar{\eta}\rangle=0.
	\end{equation}
Since $\beta(X_i,X_i)-\bar{\eta}=\rho^{-1}(\eta_i-\eta),$ then \eqref{on} becomes $$\langle \eta_i-\eta,  \eta_j-\eta\rangle=0.$$ Using Codazzi equation \eqref{codazzi3}, we see that  $[X_i,X_j]\in E_\eta^{\perp}$ for all $1 \le i ,j \le p-\ell$, which proves the last assertion. 
\end{proof}

\begin{proposition}\label{decomp}
    There exist an orthonormal subset $\{\xi_1, \ldots, \xi_{p-\ell}\} \in \Gamma(N_{f}M)$ and  $f_i \in C^{\infty}(M)$, $1\leq i\leq p-\ell$, such that $f_i$ does not vanish at any point, $1\leq i\leq p-\ell$,  $\sum_{i=1}^{p-\ell} f_i^{2}=1$ and \begin{equation}\label{beta-eta} \beta(X_i,X_i)-\bar{\eta}=f_i\xi_i.\end{equation}
\end{proposition}

\begin{proof}
    Since  $\beta$ is traceless, then \begin{equation}\label{traco} \bar{\eta}=-\frac{1}{n}\sum_{i=1}^{p-\ell}(\beta(X_i,X_i)-\bar{\eta}).\end{equation}
    Taking the inner product of both sides of the preceding equation with $\beta(X_j,X_j)-\bar{\eta}$, for each $1 \le j \le p-\ell$, and using  \eqref{on} we obtain $$||\beta(X_j,X_j)||^{2}+(n-2)\langle \beta(X_j,X_j), \bar{\eta}\rangle+(1-n)||\bar{\eta}||^{2}=0.$$ Thus \begin{equation*}\label{sum} \sum_{j=1}^{p-\ell}||\beta(X_j,X_j)||^{2}+(n-2)\langle \sum_{j=1}^{p-\ell}\beta(X_j,X_j), \bar{\eta}\rangle+(1-n)(p-\ell)||\bar{\eta}||^{2}=0.\end{equation*} Using that $||\beta||^{2}_{*}=\frac{n-1}{n}$ and that $\tr \beta=0$, the previous equation becomes $$\frac{n-1}{n}-(n-p+\ell)||\bar{\eta}||^{2}-(n-2)(n-p+\ell)||\bar{\eta}||^{2}+(1-n)(p-\ell)||\bar{\eta}||^{2}=0,$$ that is, \begin{equation}\label{eta}
		||\bar{\eta}||=\frac{1}{n}.
	\end{equation}
 Equations  \eqref{traco} and \eqref{eta} yield $\sum_{i=1}^{p-\ell}||\beta(X_i,X_i)-\bar{\eta}||^{2}=1$. Since each $\beta(X_i,X_i)$ is a simple Moebius principal normal vector field, the statements follow from   Lemma \ref{lema-ortogonal}.
\end{proof}

\begin{proposition}\label{prop:moebius-invariants}
  The following formulas hold:
    \begin{description}
\item[(i)] The Moebius second fundamental formula:
\begin{align}\label{betas}
		\beta(X_k,X_k)&=\bar{\eta}=-\frac{1}{n}\sum_{i=1}^{p-\ell}f_i \xi_i, \qquad p-\ell+1 \le k \le n, \nonumber \\
		\beta(X_i,X_i)&=\frac{n-1}{n}f_i\xi_i-\frac{1}{n}\sum_{\overset{j=1}{j \ne i}}^{p-\ell} f_j \xi_j, \quad 1 \le i \le p-\ell, \\
		\beta(X_k,X_r)&=0, \quad r \ne k \geq 1. \nonumber
	\end{align}
 \item[(ii)] The normal connection:
 \begin{align}
	\nabla^{\perp}_{X_k}\xi_i&=0, \quad p-\ell+1\le k \leq n\,\, \mbox{and} \,\, 1\le i \le p-\ell, \nonumber\\
	\nu_{ij}(X_i)&=\frac{f_i}{f_j}(\langle \grad^{*}\log f_j, X_i \rangle^{*}-\langle \delta, X_i\rangle^{*}), \quad 1 \le i \ne j \le p-\ell, \label{cn}\\
	\nu_{jk}(X_i)&=0, \quad 1 \le k \ne i \ne j \ne k \le p-\ell, \quad \mbox{where}\quad \nu_{ij}(X)=\langle \nabla^{\perp}_{X}\xi_i, \xi_j \rangle. \nonumber
\end{align}
\item[(iii)] The Moebius one-form:
\begin{align}
	\langle \omega(X_i), \xi_i \rangle&=-\frac{1}{n}(X_i(f_i)-f_i \sum_{\overset{r=1}{r \ne i}}^{p-\ell}X_i(\log f_r))+\frac{n-(p-\ell-1)}{n}f_i \langle \delta, X_i \rangle^{*}, \nonumber\\
	\langle \omega(X_i), \xi_j \rangle&=-\frac{1}{n}\left(\left(\frac{f_j^{2}+f_i^{2}}{f_j^{2}}\right)X_i(f_j)-\frac{f_i^{2}}{f_j}\langle \delta, X_i\rangle^{*}\right),\label{omegas}\\
	\langle\omega(X_k), \xi_i\rangle&=-\frac{1}{n}X_{k}(f_i), \quad 1 \le i \ne j \le p-\ell \quad \mbox{and} \quad p-\ell+1\le k \le n,  \nonumber
\end{align} 
    \end{description}where $\delta$ is the mean curvature vector field of $\Delta$ with respect to $\langle\,,\,\rangle^{*}$ and $f_1, \ldots f_{p-\ell} \in C^{\infty}(M)$, $\xi_1, \ldots, \xi_{p-\ell} \in \Gamma(N_fM)$ are given by Proposition \ref{decomp}.
\end{proposition}

\begin{proof} Equations  \eqref{betas} are immediate consequences of \eqref{beta-eta} and \eqref{traco}. Substituting  \eqref{betas} in the conformal Codazzi equation \begin{align*}
	\omega(X_i)&=(\nabla^{\perp}_{X_i}\beta)(X_k,X_k)-(\nabla^{\perp}_{X_k}\beta)(X_i,X_k)\\
	&=\nabla_{X_i}^{\perp}\bar{\eta}+\langle \nabla^{*}_{X_k}X_k, X_i\rangle^{*}(\beta(X_i,X_i)-\bar{\eta}),
\end{align*} for $1 \le i \le p-\ell$ and $p-\ell+1 \le k \le n$,  we obtain $$ \omega(X_i)=-\frac{1}{n}\sum_{r=1}^{p-\ell}(X_i(f_r)\xi_r+f_r\nabla^{\perp}_{X_i}\xi_r)+\langle \delta, X_i\rangle^{*}f_i\xi_i,$$ hence \begin{equation}\label{omegas1}\begin{cases}
		\langle \omega(X_i), \xi_i\rangle=-\frac{1}{n}(X_i(f_i)+\sum_{\overset{r=1}{r \ne i}}^{p-\ell}f_r\nu_{ri}(X_i))+f_i\langle \delta, X_i\rangle^{*},\vspace{5pt}\\
		\langle \omega(X_i), \xi_j \rangle=-\frac{1}{n}(X_i(f_j)+\sum_{\overset{r=1}{r \ne j}}^{p-\ell}f_r \nu_{rj}(X_i)), \quad i \ne j. \end{cases}
\end{equation} 
On the other hand, substituting \eqref{betas} in the conformal Codazzi equation \begin{align*}
	\omega(X_i)&=(\nabla^{\perp}_{X_i}\beta)(X_j,X_j)-(\nabla^{\perp}_{X_j}\beta)(X_i,X_j)\\
	&=\nabla^{\perp}_{X_i}\beta(X_j,X_j)+\langle \nabla^{*}_{X_j}X_j,X_i\rangle^{*}(\beta(X_i,X_i)-\beta(X_j,X_j))
\end{align*} for $1 \le i \ne j \le p-\ell$, yields $$ \omega(X_i)=-\frac{1}{n}(\sum_{\overset{r=1}{r \ne j}}^{p-\ell}X_i(f_r)\xi_r+\sum_{r=1}^{p-\ell}f_r \nabla^{\perp}_{X_i}\xi_r)+\frac{n-1}{n}X_i(f_j)\xi_j+f_j\nabla^{\perp}_{X_i}\xi_j+\langle \nabla^{*}_{X_j}X_j, X_i \rangle^{*}(f_i\xi_i-f_j\xi_j),$$  that is, \begin{equation}\label{omegas2}
	\begin{cases}
		\langle \omega(X_i), \xi_i\rangle=-\frac{1}{n}(X_i(f_i)-nf_j\nu_{ji}(X_i)+\sum_{\overset{r=1}{r \ne i}}^{p-\ell}f_r \nu_{ri}(X_i))+f_i\langle \nabla^{*}_{X_j}X_j,X_i\rangle^{*},\\
		\langle \omega(X_i), \xi_j\rangle=\frac{1}{n}((n-1)X_i(f_j)-\sum_{\overset{r=1}{r \ne j}}^{p-\ell}f_r \nu_{rj}(X_i))-f_j\langle \nabla^{*}_{X_j}X_j,X_i\rangle^{*},\\
		\langle \omega(X_i), \xi_k \rangle=-\frac{1}{n}(X_i(f_k)+\sum_{\overset{r=1}{r \ne k}}^{p-\ell}f_r \nu_{rk}(X_i))+f_j\nu_{jk}(X_i),
	\end{cases}
\end{equation} for all $1 \le k \ne i \ne j \ne k \le p-\ell$. 
Comparing the expressions in \eqref{omegas1} and  \eqref{omegas2}, we obtain \begin{align}
	\langle \nabla^{*}_{X_j}X_j,X_i\rangle^{*}=&\langle \grad^{*}\log f_j, X_i \rangle^{*},\label{value-nabla}\\
	\nu_{ij}(X_i)=&\frac{f_i}{f_j}(\langle \nabla^{*}_{X_j}X_j, X_i\rangle^{*}-\langle \delta, X_i\rangle^{*}),\label{nuij}\\
	\nu_{jk}(X_i)=&0, \,\,\, 1 \le k \ne i \ne j \ne k \le p-\ell.\label{nujk}
\end{align} 
The conformal Codazzi equation $$ (\nabla^{\perp}_{X_k}\beta)(X_r,X_k)-(\nabla^{\perp}_{X_r}\beta)(X_k,X_k)=\omega((X_k \wedge X_r)X_k),$$ $ p-\ell+1 \le k \ne r \le n$, implies \begin{equation}\label{omega1} \omega(X_r)= {\nabla^{\perp}_{X_r}\bar{\eta}},\end{equation}
whereas $$({\nabla^{\perp}_{X_k}}\beta)(X_i,X_i)-(\nabla^{\perp}_{X_i}\beta)(X_k,X_i)=\omega((X_k \wedge X_i)X_i)= {{\nabla^{\perp}_{X_k}}}\bar{\eta},$$ for $p-\ell+1 \le k \le n$ and $1 \le i \le p-\ell$, yields $$ \nabla^{\perp}_{X_k}(\beta(X_i,X_i)-\bar{\eta})=\langle \nabla^{*}_{X_i}X_i, X_k \rangle^{*}(\beta(X_i,X_i)-\bar{\eta}) \Longleftrightarrow \nabla^{\perp}_{X_k}(f_i \xi_i)=\langle \nabla^{*}_{X_i}X_i,X_k\rangle^{*}f_i\xi_i.$$ Therefore \begin{equation}\label{nablas}
    \langle \nabla^{*}_{X_i}X_i,X_k\rangle^{*}=\langle \grad^{*}\log f_i, X_k\rangle^{*} \quad \mbox{and} \quad \nabla^{\perp}_{X_k}\xi_i=0, 
\end{equation} for all  $p-\ell+1 \le k \le n$ and $1 \le i \le p-\ell.$ It follows from 
\eqref{nuij}, \eqref{nujk} and \eqref{nablas} that \eqref{cn} holds, and substituting it in \eqref{omegas1} and \eqref{omega1} gives \eqref{omegas}.
\end{proof}

\begin{proposition}\label{prop-totally-geodesic}
    The  distribution $\Delta^{\perp}$ is totally geodesic with respect to $\langle \,,\,\rangle^{*}$ if and only if $(\grad^{*}f_i)_{\Delta}=0$ for all $i=1, \ldots, p -\ell.$
\end{proposition}

\begin{proof} The conformal Codazzi equation $ (\nabla^{\perp}_{X_i}\beta)(X_j,X_k)=(\nabla^{\perp}_{X_j}\beta)(X_i,X_k)$, for $1 \le i \ne j \le p-\ell$ and $p-\ell+1 \le k \le n$, implies that  \begin{align*}
   \langle \nabla^{*}_{X_i}X_j,X_k\rangle^{*}f_{j}\xi_j=\langle \nabla^{*}_{X_j}X_i,X_k\rangle^{*}f_i\xi_i.
\end{align*} Given that $\xi_i,\xi_j$ are linearly independent  and that $f_i,f_j$ do not vanish at any point,  then \begin{equation}\label{flat}
    (\nabla_{X_i}^{*}X_j)_{\Delta}=0=(\nabla^{*}_{X_j}X_i)_{\Delta}.
\end{equation} 
The statement then follows from  \eqref{nablas} and \eqref{flat}. 
\end{proof}

\begin{corollary}\label{tp-net}
The net $\mathcal{E}=\{\spa\{X_1\}, \ldots, \spa\{X_{p-\ell}\}, \Delta\}$ in $(M^{n}, \langle\,,\,\rangle^{*})$ is a  $\TP$-net. 
\end{corollary}

\begin{proof}
    By \eqref{value-nabla} and \eqref{nablas} we have \begin{equation*}\label{values-nabla}
    \nabla^{*}_{X_i}X_i=\grad^{*}\log f_i-\langle \grad^{*}\log f_i, X_i\rangle^{*}X_i, \quad 1 \le i \le p-\ell,
\end{equation*} hence
 \begin{equation*}
    \langle \nabla^{*}_{X_i}X_i, X_k \rangle^{*}=\langle \grad^{*}\log f_i, X_k \rangle^{*}=-\langle \grad^{*}\log f_i^{-1}, X_k\rangle^{*},
\end{equation*} for all $ 1\le k \le n$ with $k \ne i.$ Therefore, $\spa\{X_i\}$ is an umbilical distribution in $(M^{n}, \langle \,,\,\rangle^{*})$ with mean curvature vector field  $-(\grad^{*}\log f_i^{-1})_{(\spa\{X_i\})^{\perp}}$ for all $1 \le i \le p-\ell$. That $\Delta^{\perp}$ is integrable has been shown in Lemma \ref{lema-ortogonal}.

We claim $(\spa\{X_i\})^{\perp}$ is integrable for any $1 \le i \le p-\ell$.  
Indeed, as we know that $\Delta$ is umbilical, then $[X_j,X_k] \in \Delta \subset (\spa\{X_i\})^{\perp}$ for $p-\ell+1 \le j, k \le n$. On the other hand, for all $1 \le j \le p-\ell$ and $p-\ell+1 \le k \le n$, with $j \ne i$, it follows  from \eqref{flat} and the conformal Codazzi equation $(\nabla^{\perp}_{X_i}\beta)(X_k,X_j)=(\nabla^{\perp}_{X_k}\beta)(X_i,X_j)$ that
$$ \langle \nabla^{*}_{X_k}X_j,X_i \rangle^{*}(\beta(X_i,X_i)-\beta(X_j,X_j))=0, $$ 
hence $[X_j,X_k] \in (\spa\{X_i\})^{\perp}$.

Finally, assume $p-\ell \geq 3$ and let $1 \le k \ne j \ne i\ne k\le p-\ell$. From the conformal Codazzi equation $$(\nabla^{\perp}_{X_k}\beta)(X_j,X _i)=(\nabla^{\perp}_{X_j}\beta)(X_k,X_i),$$ we obtain \begin{equation*}
    \langle \nabla^{*}_{X_k}X_i, X_j \rangle^{*}(\beta(X_j,X_j)-\beta(X_i,X_i))=\langle \nabla^{*}_{X_j}X_i, X_k\rangle^{*} (\beta(X_k,X_k)-\beta(X_i,X_i)),
\end{equation*} which is equivalent to \begin{equation*}
    \langle \nabla^{*}_{X_k}X_i,X_j\rangle^{*}(f_j \xi_j-f_i\xi_i)=\langle \nabla^{*}_{X_j}X_i, X_k \rangle^{*}(f_k \xi_k-f_i \xi_i).
\end{equation*} This easily  implies that  $ \langle \nabla^{*}_{X_j}X_k, X_i \rangle^{*}=0=\langle \nabla^{*}_{X_k}X_j, X_i \rangle^{*},$ thus proving our claim.
\end{proof}

\section{The main result}\label{class}
   
    We are now in a position to state and prove the main result of this article.
    
\begin{theorem}\label{theorem-main} Let $f:M^{n}\to \mathbb{R}^{n+p}$,  $n \geq 5$ and  $2p \leq n$, be  a proper isometric immersion with flat normal bundle and constant Moebius curvature $c$ with at least three distinct principal normal vector fields. Then $f(M^{n})$ is the image by a conformal transformation of $\mathbb{R}^{n+p}$ of an open subset of a submanifold as in one of Examples \ref{cilindro}, which is determined by a submanifold $g:U \subset \mathbb{Q}^{p-\ell}_{\tilde{c}}\to \mathbb{Q}^{2p-\ell}_{\tilde{c}}$ with the property that the norm of its mean curvature vector field is given as in Corollary \ref{meancurvature}, with $0 \le \ell \le p-2$ and $\tilde{c}=0,1$ or $-1,$ respectively.
\end{theorem} 
\begin{proof}  First notice that  $(M^{n}, \langle \,,\,\rangle_{f})$  is conformally flat, for  $(M^{n}, \langle \,,\,\rangle^{*})$ has constant curvature and $\langle \,,\,\rangle_{f}$ and 
    $\langle \,,\,\rangle^{*}$ are conformal. Since the assumptions on $n$ and $p$ imply that $p\leq n-3$, by Theorem \ref{theorem-moore} there exists a principal normal vector field $\eta$ with multiplicity $n-p+\ell$ for some $0 \le \ell \le p-2$. Moreover,    the remaining $p-\ell$ principal normal vector fields of $f$ are all simple by Theorem \ref{theorem-DOV}. Since  the Moebius form of an isometric immersion with constant Moebius curvature and flat normal bundle is closed by Lemma \eqref{moebius-closed}, the Moebius invariants of $f$ are given as in Proposition \eqref{prop:moebius-invariants}.
Let $X_1, \ldots X_n$ be an orthonormal frame with respect to the metric $\langle\,,\,\rangle^{*}$  that diagonalizes $\beta$ and $\psi$ simultaneously, and such that
 \begin{equation}\label{iieta}\beta(X_i,X_i)=\bar{\eta},\end{equation} for all $ p-\ell+1 \le i  \le n$, where $\bar{\eta}$ is the Moebius principal normal vector field associated with $\eta$.  Consider the smooth distributions \begin{equation}
    \Delta=\spa\{X_i: \, p-\ell+1 \le i \le n\} \quad \mbox{and} \quad \Delta^{\perp}=\spa\{X_i: \, 1 \le i \le p-\ell\}. 
\end{equation} Corollary \ref{tp-net} shows that $\mathcal{E}$ is a $\TP$-net. By Theorem \ref{decomposition-theorem},  at each $x \in M^{n}$ there exists an open subset $U$ of $M^{n}$ containing $x$ and a product representation $\Phi: \prod_{k=1}^{p-\ell+1}M_k \to U$ of $\mathcal{E}$ that is an isometry with respect to a twisted product metric \begin{equation}\label{metric-twisted}\langle\,,\,\rangle=\sum_{k=1}^{p-\ell+1}\rho_k^{2}\pi_k^{*}\langle\,,\,\rangle_k \end{equation} on $\prod_{k=1}^{p-\ell+1}M_k$, for some \emph{twisting functions}  $\rho_k \in C^{\infty}(\prod_{k=1}^{p-\ell+1}M_k)$, $1\leq k\leq p-\ell +1$ .

Let $(E_{i})_{i=1, \ldots,p-\ell+1}$ be the product net of $\prod_{k=1}^{p-\ell+1}M_k$, that is,  $E_{i}(x):={\tau_i^{x}}_{*}T_{x_i}M_i$ for all  $x=(x_1, \ldots, x_{p-\ell+1})\in \prod^{p-\ell+1}_{k=1}M_k$, and let $\tau_i^{x}:M_i \to \prod_{k=1}^{p-\ell+1}M_k$ be the standard inclusion. As observed in \eqref{umbilical}, $E_i$ is an umbilical distribution with mean curvature vector field $-(\grad (\log \circ \rho_i))_{E_i^{\perp}}$, where  $\grad$ is the gradient with respect to $\langle\,,\,\rangle.$

We claim that there exist a coordinate system $(\Tilde{x}_1, \ldots, \Tilde{x}_n)$ on $\prod_{k=1}^{p-\ell+1}M_k$ and functions $r_i=r_i(\Tilde{x}_i)$, $ 1 \le i \le p-\ell$, such that $\rho_i=r_i(\Tilde{x}_i)f_i^{-1}\circ \Phi$. Indeed, as shown in the proof of Corollary \ref{tp-net}, $\spa \{X_i\}$ is an umbilical distribution on $(M^{n}, \langle \,,\,\rangle^{*})$ with mean curvature vector field $-(\grad^{*}\log f_i^{-1})_{(\spa\{X_i\})^{\perp}}$ for any $1 \le i \le p-\ell$. Hence  \begin{align*}
    -\langle \grad^{*}\log  f_i^{-1}, X_k \rangle^{*}=\langle \nabla^{*}_{X_i}X_i, X_k \rangle^{*}&=\langle \nabla^{*}_{{\tau_i^{x}}_{*x_i}\bar{X}_i}\Phi_{*}{\tau_i^{x}}_{*}\bar{X}_i, \Phi_{*}{\tau_{k}^{x}}_{* x_k}\bar{X}_k\rangle^{*}\\
    &=\langle \Phi_{*}\nabla_{{\tau_{i}^{x}}_{*x_i}\bar{X}_i}{\tau_{i}^{x}}_{*}\bar{X}_i, \Phi_{*}{\tau_{k}^{x}}_{*x_k}\bar{X}_k \rangle^{*}\\
    &=\langle \nabla_{{\tau_{i}^{x}}_{*x_i}\bar{X}_i}{\tau_{i}^{x}}_{*}\bar{X}_i, {\tau_{k}^{x}}_{*x_k}\bar{X}_k \rangle\\
    &=-\langle \grad (\log \circ \rho_i), {\tau_{k}^{x}}_{*x_k}\bar{X}_k\rangle,
\end{align*} for all $k \ne i$ with $1 \le k \le n.$ Thus, \begin{align}\label{igual}
    (\Phi_{*}\grad(\log \circ \rho_i))_{(\Phi_{*}E_i(x))^{\perp}}&=(\grad^{*}(\log  f_{i}^{-1}))_{(\spa\{X_i\}(\Phi(x)))^{\perp}}\nonumber\\
    &=(\Phi_{*}\grad(\log \circ f_{i}^{-1}\circ \Phi))_{(\Phi_{*}E_i(x))^{\perp}}.
\end{align}

Introduce coordinates $(\Tilde{x}_1, \ldots,\Tilde{x}_{p-\ell}, \ldots, \Tilde{x}_{n})$ on $\prod_{k=1}^{p-\ell+1}M_k$. Using \eqref{igual}, we have $$ \rho_i=r_i(\Tilde{x}_i)f_i^{-1}\circ \Phi,$$ for some smooth functions $r_i=r_i(\Tilde{x}_i)$ with $1\le i \le p-\ell.$ Defining ${x}_i=\int r_i(\Tilde{x}_i)d\Tilde{x}_i$,  with respect to the coordinates $(x_1, \ldots, x_{p-\ell}, \Tilde{x}_{p-\ell+1}, \ldots, \Tilde{x}_n)$ the metric \eqref{metric-twisted} can be written as  \begin{equation}\label{metric-twisted2}
    \langle\,,\,\rangle=\sum_{i=1}^{p-\ell}f_i^{-2}dx_i^{2}+\varphi^{2}\sum_{i,j=p-\ell+1}^{n}\Tilde{g}_{ij}d\Tilde{x}_id\Tilde{x}_j,
\end{equation} for some   $\varphi \in C^{\infty}(U)$, where $\Tilde{g}_{ij}$ are the coefficients of the metric $\langle\,,\,\rangle_{p-\ell+1}$ (here we omit $\Phi$ for the sake  of simplicity).

  Now we show that we can choose orthogonal coordinates on $M_{p-\ell+1}$ with respect to $\langle\, , \,\rangle_{p-\ell+1}$. Indeed,  for a fixed choice $\bar{x}:=(\bar{x}_1,\ldots, \bar{x}_{p-\ell})$ of the coordinates $(x_1, \ldots, x_{p-\ell})$, let $\tau_{p-\ell+1}^{\bar{x}}:M_{p-\ell+1}\to \prod_{k=1}^{p-\ell+1}M_k$ be defined by $\tau_{p-\ell+1}^{\bar{x}}(x_{p-\ell+1})=(\bar{x}_1, \ldots, \bar{x}_{p-\ell}, x_{p-\ell+1})$. Then $\langle {\tau_{p-\ell+1}^{\bar{x}}}_{*}v, {\tau_{p-\ell+1}^{\bar{x}}}_{*}w\rangle=\varphi_{\bar{x}}^{2}\langle v,w\rangle_{p-\ell+1}$ for all $x_{p-\ell+1}\in M_{p-\ell+1}$ and all $v,w\in T_{x_{p-\ell+1}}M_{p-\ell+1}$, where $\varphi_{\bar{x}}\colon M_{p-\ell+1} \to \mathbb{R}_{+}$ is given by $$\varphi_{\bar{x}}(x_{p-\ell+1})=\varphi(\bar{x}_1,\ldots,\bar{x}_{p-\ell},  x_{p-\ell+1}).$$ Thus $\tau_{p-\ell+1}^{\bar{x}}$ is a conformal diffeomorphism of $M_{p-\ell+1}$ onto the leaf $L(\bar{x}):={\tau_{p-\ell+1}^{\bar{x}}}(M_{p-\ell+1})$ of $\Delta$ with conformal factor $\varphi_{\bar{x}}$.

Since $L(\bar{x})$ is umbilical in $(M^{n}, \langle\,,\,\rangle^{*})$ with mean curvature vector field $\delta(\bar{x}, x_{p-\ell+1})=
    (\grad^{*}\log \varphi)_{\Delta^{\perp}}$, by the Gauss equation the metric $g_{\bar{x}}$  in $L(\bar{x})$ induced by $\langle\,,\,\rangle^{*}$ has  curvature \begin{align*}
   c(\bar{x})&=c+||(\grad^{*}\log \varphi)_{\Delta^{\perp}}||^{2}\\
    &=c+\sum_{i=1}^{p-\ell}(X_i(\log \varphi))^{2}\\
    &=c+\sum_{i=1}^{p-\ell}(f_i(\bar{x},x_{p-\ell+1}))^{2}(\frac{\partial \log \varphi}{\partial x_i})^{2}.
\end{align*} 

Since $\langle\,,\,\rangle^{*}$ has constant curvature, then   $\delta$ is parallel in the normal connection of the inclusion of $L(\bar{x})$ into  $(M^{n}, \langle\,,\,\rangle^{*})$. In particular, it has constant length, and hence, the metric ${\tau_{p-\ell+1}^{\bar{x}}}^{*}g_{\bar{x}}=\varphi^{2}_{\bar{x}}\langle\,,\,\rangle_{p-\ell+1}$ has constant curvature $c(\bar{x})$. In particular,  there exists local orthogonal coordinates $(x_{p-\ell+1}, \ldots, x_n)$ on $(M_{p-\ell+1}, {\tau^{\bar{x}}_{p-\ell+1}}^{*}g_{\bar{x}})$.  Since  ${\tau_{p-\ell+1}^{\bar{x}}}^{*}g_{\bar{x}}=\varphi^{2}_{\bar{x}}\langle\,,\,\rangle_{p-\ell+1}$,  $(x_{p-\ell+1}, \ldots, x_n)$ are also orthogonal coordinates on $(M_{p-\ell+1}, \langle\,,\,\rangle_{p-\ell+1})$. Write $$ \langle \,,\,\rangle_{p-\ell+1}=\sum_{i=p-\ell+1}^{n}V_i^{2}dx_i^{2},$$ with $V_i=\sqrt{\langle \frac{\partial}{\partial x_i}, \frac{\partial}{\partial x_i}\rangle_{p-\ell+1}}$ for $i \geq p-\ell+1.$ 
Then the metric \eqref{metric-twisted2} takes the form \begin{equation}\label{metric-moebius}
     \langle\,,\,\rangle^{*}=\sum_{i=1}^{p-\ell}f_{i}^{-2}dx_i^{2}+\varphi^{2}\sum_{i=p-\ell+1}^{n}V_i^{2}dx_i^{2}.
 \end{equation}

  Denote $$ h_{ij}=\frac{1}{v_i}\frac{\partial v_j}{\partial x_i}, \quad 1 \le i \ne j \le n,$$ with $v_i=f_i^{-1}$ for $1 \le i \le p-\ell$ and $v_k=\varphi V_k$ for $k \geq p-\ell+1.$
Then \begin{equation*}
    h_{ij}=\begin{cases}-\dfrac{f_i}{f_j^{2}}\dfrac{\partial f_j}{\partial x_i}, \qquad \qquad 1 \le i \ne j \le p-\ell,\vspace{5pt}\\
    f_iV_j\frac{\partial \varphi}{\partial x_i}, \qquad \qquad  1 \le i \le p-\ell \,\, \mbox{and} \,\, j \geq p-\ell+1, \vspace{5pt}\\
    -\dfrac{1}{\varphi V_i}f_{j}^{-2}\dfrac{\partial f_j}{\partial x_i}, \quad \quad i \geq p-\ell+1 \,\, \mbox{and} \,\, 1 \le j \le p-\ell,\\
    \frac{V_j}{V_i}\frac{\partial \log \varphi}{\partial x_i}+H_{ij}, \qquad i \ne j \geq p-\ell+1,\,\,\,\mbox{with}\,\,\,\, H_{ij}:=\frac{1}{V_i}\frac{\partial V_j}{\partial x_i}.
    \end{cases}
\end{equation*} 
It is well-known that the conditions for \eqref{metric-moebius} to have constant sectional curvature $c$ are \begin{equation*}
    \begin{cases}
    \dfrac{\partial h_{ij}}{\partial x_i}+\dfrac{\partial h_{ji}}{\partial x_j}+\sum_{k \ne i,j}h_{ki}h_{kj}+cv_iv_j=0,\\
    \dfrac{\partial h_{ik}}{\partial x_j}=h_{ij}h_{jk}, \quad 1 \le i \ne k \ne j \ne i \le n.
    \end{cases} 
\end{equation*}
Let us compute $\frac{\partial h_{ik}}{\partial x_j}=h_{ij}h_{jk}$ for $p-\ell+1 \le i \ne j \le n$ and $1 \le k \le p-\ell.$ We have \begin{align*}
    \frac{\partial h_{ik}}{\partial x_j}=\frac{1}{(\varphi V_i)^{2}}\frac{\partial (\varphi V_i)}{\partial x_j}f_k^{-2}\frac{\partial f_k}{\partial x_i}+\frac{2}{\varphi V_i}f_{k}^{-3}\frac{\partial f_k}{\partial x_j}\frac{\partial f_k}{\partial x_i}-\frac{1}{\varphi V_i}f_k^{-2}\frac{\partial^{2}f_k}{\partial x_j \partial x_i}
\end{align*} and \begin{align*}
    h_{ij}h_{jk}=(\frac{V_j}{V_i}\frac{\partial \log \varphi}{\partial x_i}+\frac{1}{V_i}\frac{\partial V_j}{\partial x_i})(-\frac{1}{\varphi V_j}f_{k}^{-2}\frac{\partial f_k}{\partial x_j})=-\frac{1}{\varphi V_i}\frac{\partial \log (\varphi V_j)}{\partial x_i}f_k^{-2}\frac{\partial f_k}{\partial x_j}.
\end{align*} 
Multiplying the previous expression by $\varphi V_if_k^{2}$, we obtain \begin{equation}\label{ikj} \frac{\partial^{2}f_k}{\partial x_j \partial x_i}=\frac{\partial \log (\varphi V_i)}{\partial x_j}\frac{\partial f_k}{\partial x_i}+\frac{\partial \log (\varphi V_j)}{\partial x_i}\frac{\partial f_k}{\partial x_j}+2f_k^{-1}\frac{\partial f_k}{\partial x_j}\frac{\partial f_k}{\partial x_i}.\end{equation}
Since $\sum_{k=1}^{p-\ell}f_{k}^{2}=1$, then \begin{equation}\label{cond1} \sum_{k=1}^{p-\ell}f_k\frac{\partial f_k}{\partial x_i}=0 \quad \mbox{and} \quad \sum_{k=1}^{p-\ell}f_k\frac{\partial^{2}f_k}{\partial x_j \partial x_i}=-\sum_{k=1}^{p-\ell}\frac{\partial f_k}{\partial x_j}\frac{\partial f_k}{\partial x_i}.\end{equation} It follows from \eqref{ikj} and \eqref{cond1} that \begin{equation}\label{cond2} \sum_{k=1}^{p-\ell}\frac{\partial f_k}{\partial x_j}\frac{\partial f_k}{\partial x_i}=0,\,\,\,i\ne j \geq p-\ell+1.\end{equation}

Define $ W=(f_1, \ldots, f_{p-\ell}) \quad \mbox{and} \quad U_i=\left(\frac{\partial f_1}{\partial x_i}, \ldots, \frac{\partial f_{p-\ell}}{\partial x_i}\right)$,  $p-\ell+1 \le i \le n.$ By Proposition \ref{prop-totally-geodesic},  $\Delta^{\perp}$ is totally geodesic with respect to $\langle\,,\,\rangle^{*}$ if and only if $ U_i=0$ for all $ p-\ell+1 \le i \le n.$ Equations \eqref{cond1} and \eqref{cond2} become \begin{align*}
    \langle W, U_i \rangle&=0, \quad p-\ell+1 \le i \le n, \vspace{12pt}\\
    \langle U_i,U_j\rangle&=0, \quad p-\ell+1 \le i \ne j \le n,
\end{align*} that is, $W, U_{p-\ell+1}, \ldots, U_n$ are $n-p+\ell+1$ orthogonal vectors in $\mathbb{R}^{p-\ell}.$

Since $n \geq 2p$ by assumption, then $n-p+\ell+1>p-\ell,$ that is, $n>2(p-\ell)-1$.  Thus, we can assume that $U_{j}=0$ for $n-p+\ell+1 \le j \le n.$  This means that  $f_1, \ldots, f_{p-\ell}$ do not depend on $x_j$ for $n-p+\ell+1 \le j \le n.$
Assume by contradiction that $\Delta^{\perp}$ is not totally geodesic with respect to $\langle \,,\,\rangle^{*}$. 
Then there exist $x \in M^{n}$ and $j \geq p-\ell+1$,  say,  $j=p-\ell+1$, such that $U_{p-\ell+1}(x)\ne 0$. Assume, without loss of generality, that $\frac{\partial f_1}{\partial x_{p-\ell+1}} \ne 0$ on a neighborhood of $x.$  Applying \eqref{ikj} for $j=p-\ell+1$, $k=1$ and $i \geq n-p+\ell+1$, we obtain $$ \frac{\partial \log (\varphi V_{p-\ell+1})}{\partial x_i}=0,$$ that is, $v_{p-\ell+1}:=\varphi V_{p-\ell+1}$ does not depend on $x_i$. We can write $$ \langle \,,\,\rangle^{*}=\sum_{i=1}^{p-\ell}f_i^{-2}dx_i^{2}+\varphi^{2}\sum_{i=p-\ell+1}^{n-p+\ell}V_i^{2}dx_i^{2}+(\varphi V_{p-\ell+1})^{2}V_{p-\ell+1}^{-2}\sum_{i=n-p+\ell+1}^{n}V_i^{2}dx_i^{2}.$$

Denote $g_1=\sum_{i=1}^{p-\ell}f_i^{-2}dx_i^{2}+\varphi^{2}\sum_{i=p-\ell+1}^{n-p+\ell}V_i^{2}dx_i^{2}$ and $g_2=V_{p-\ell+1}^{-2}\sum_{i=n-p+\ell+1}^{n}V_i^{2}dx_i^{2}$, so that $$ \langle \,,\,\rangle^{*}=g_1+v_{p-\ell+1}^{2}g_2.$$ 
The fact that $ \langle \,,\,\rangle^{*}$ has constant sectional curvature $c$ is well-known 
to be equivalent to the following conditions:
\begin{enumerate}
    \item[(a)]$g_1$ has constant curvature $c$.
    \item[(b)] $\Hess \, v_{p-\ell+1}+cv_{p-\ell+1} g_1=0$
    \item[(c)] $g_2$ has constant curvature $||\grad \, v_{p-\ell+1}||^{2}+c v_{p-\ell+1}^{2}$,
\end{enumerate} where $\Hess$ and $\grad $ are computed with respect to $g_1$. In particular,
\begin{equation}\label{pii}\Hess \, v_{p-\ell+1}(\partial_i, \partial_i)+cv_{p-\ell+1}v_i^{2}=0, \quad 1 \le i \le p-\ell.
\end{equation} 
We have
\begin{align*}
    \nabla_{\partial_i}\partial_i=\nabla_{\partial_i}(v_iX_i)&=\frac{\partial v_i}{\partial x_i}X_i+v_i\left(\sum_{\overset{j=1}{j \ne i}}^{n-p+\ell}\langle \nabla_{\partial_i}X_i, X_j\rangle X_j\right)\\
    &=\frac{\partial v_i}{\partial x_i}X_i-v_i\sum_{\overset{j=1}{j \ne i}}^{n-p+\ell}h_{ji}X_j\\
    &=\frac{\partial v_i}{\partial x_i}\frac{1}{v_i}\frac{\partial}{\partial x_i}-v_i\sum_{\overset{j=1}{j \ne i}}^{n-p+\ell}h_{ji}\frac{1}{v_j}\frac{\partial}{\partial x_j}.
\end{align*} Hence  \begin{align*}
    \Hess\, v_{p-\ell+1}(\partial_i, \partial_i)&=\frac{\partial}{\partial x_i}(v_ih_{i(p-\ell+1)})-(\nabla_{\partial_i}\partial_i)(v_{p-\ell+1})\\
    &=\frac{\partial v_i}{\partial x_i}h_{i(p-\ell+1)}+v_i\frac{\partial h_{i(p-\ell+1)}}{\partial x_i}-\frac{1}{v_i}\frac{\partial v_i}{\partial x_i}\frac{\partial v_{p-\ell+1}}{\partial x_i}+v_i\sum_{\overset{j=1}{j \ne i}}^{n-p+\ell}h_{ji}\frac{1}{v_j}\frac{\partial v_{p-\ell+1}}{\partial x_j}\\
    &=v_i\left(\frac{\partial h_{i(p-\ell+1)}}{\partial x_i}+\sum_{\overset{j=1}{j \ne i,p-\ell+1}}^{n-p+\ell}h_{ji}h_{j(p-\ell+1)}+h_{(p-\ell+1)i}\frac{1}{v_{p-\ell+1}}\frac{\partial v_{p-\ell+1}}{\partial x_{p-\ell+1}}\right).
\end{align*} Thus, \eqref{pii} is equivalent to \begin{equation}\label{piib}
    \frac{\partial h_{i(p-\ell+1)}}{\partial x_i}+\sum_{\overset{j=1}{j \ne i,p-\ell+1}}^{n-p+\ell}h_{ji}h_{j(p-\ell+1)}+h_{(p-\ell+1)i}\frac{1}{v_{p-\ell+1}}\frac{\partial v_{p-\ell+1}}{\partial x_{p-\ell+1}}+cv_{p-\ell+1}v_i=0.
\end{equation}
On the other hand, since $g_1=\sum_{i=1}^{n-p+\ell}v_i^{2}dx_i^{2}$ has constant sectional curvature $c$, then
$$\frac{\partial h_{i(p-\ell+1)}}{\partial x_i}
+\frac{\partial h_{(p-\ell+1)i}}
{\partial x_{(p-\ell+1)i}}+\sum_{\overset{j=1}{j \ne i,p-\ell+1}}^{n-p+\ell}h_{ji}h_{j(p-\ell+1)}+cv_iv_{p-\ell+1}=0.$$
Comparing with \eqref{piib} yields
 \begin{equation*}
    \frac{\partial h_{(p-\ell+1)i}}{\partial x_{p-\ell+1}}=h_{(p-\ell+1)i}\frac{1}{v_{p-\ell+1}}\frac{\partial v_{p-\ell+1}}{\partial x_{p-\ell+1}},\end{equation*}
    or equivalently,
 \begin{equation*} \frac{\partial}{\partial x_{p-\ell+1}}(h_{(p-\ell+1)i}/v_{p-\ell+1})=0.
\end{equation*}  Using that $h_{(p-\ell+1)i}=-\frac{1}{v_{p-\ell+1}}f_{i}^{-2}\frac{\partial f_{i}}{\partial x_{p-\ell+1}}$ for $1 \le i \le p-\ell$, we obtain  \begin{equation*}
    \frac{1}{v_{p-\ell+1}^{2}}f_i^{-2}\frac{\partial f_i}{\partial x_{p-\ell+1}}=\phi_i(x_1, \ldots, x_{p-\ell}, x_{p-\ell+2}, \ldots, x_{n-p+\ell}),
\end{equation*} for some function $\phi_i=\phi_i(x_1, \ldots, x_{p-\ell}, x_{p-\ell+2}, \ldots, x_{n-p+\ell}).$
In particular, for $i=1$, the preceding equation yields a contradiction, for   $\frac{\partial f_1}{\partial x_{p-\ell+1}}\ne 0 $  at $x$ by hypothesis, whereas its right hand side does not depend on $x_{p-\ell+1}$. 

   We conclude that $\Delta^{\perp}$ is totally geodesic with respect to $\langle\,,\,\rangle^{*}$, hence  spherical with respect to the metric induced by $f$ by Proposition \ref{totally-umbilical}, with $(\grad_{f} \log \rho)_{\Delta}$ as its mean curvature vector field.

Since $\Delta=E_{\eta}$, Theorem \ref{cor:dupinpcn} shows that $f$ is locally, up to a composition with a conformal transformation of $\mathbb{R}^{n+p}$, an immersion $f=\Theta \circ (g, \Id):M^{p-\ell}\times \mathbb{Q}^{n-p +\ell}_{-\Tilde{c}} \to \mathbb{R}^{n+p}$, where $g:M^{p-\ell}\to \mathbb{Q}^{2p-\ell}_{\Tilde{c}}$, $0 \le \ell \le p-2,$ is an isometric immersion with flat normal bundle and $\Theta=\Id:\mathbb{R}^{n+p}\to \mathbb{R}^{n+p}$ if $\Tilde{c}=0$ or $\Theta$ is as in Examples \ref{cilindro}-$(ii)$ and $(iii)$  if $\Tilde{c} \ne 0.$  Now, as $f$ is proper then so is $g$ (because, it follows from Proposition 5 of \cite{DOV} that the property of an isometric immersion with flat normal bundle being proper is invariant under conformal changes of the ambient metric), so then, $\nu^{g}$ vanishes identically  on $M^{n}$, otherwise, if $\nu^{g}(y) \geq 1$ for some $y \in M^{n}$ then $g$ would have at least one of the $p-\ell$ principal normal vector fields vanishing, which implies $\eta$ would have multiplicity strictly greater than $n-p+\ell$, what it is a contradiction.  Since $f$ has constant Moebius curvature $c$, the conclusion of this theorem is a consequence of Lemma \ref{moebius-constante} and Corollary \ref{meancurvature}.
\end{proof}

\begin{Remarks} $(i)$ \emph{Theorem \eqref{theorem-main} together with Theorem \eqref{theorem-main-2} provide a complete classification of all isometric immersions $f:M^{n}\to \mathbb{R}^{n+p}$, $n\geq 2p$ and $n \geq 5$, with constant Moebius curvature and flat normal bundle.\vspace{1ex}\\
$(ii)$     The proof of Theorem \ref{theorem-main}  also works for isometric immersions $f:M^{4}\to \mathbb{R}^{6}$ with constant Moebius curvature and flat normal bundle that have a principal normal vector field $\eta$ of multiplicity $2$ (observe that the case in which $\eta$ has multiplicity $3$ is included in Theorem \ref{theorem-main-2}).  Therefore, for the classification of all submanifolds $f:M^{n}\to \mathbb{R}^{n+2}$, $n \geq 4,$ with constant Moebius curvature and flat normal bundle, it remains to study the case in which there are   four distinct principal normal vector fields. }    
\end{Remarks}

\noindent Universidade de S\~ao Paulo\\
Instituto de Ci\^encias Matem\'aticas e de Computa\c c\~ao.\\
Av. Trabalhador S\~ao Carlense 400\\
13566-590 -- S\~ao Carlos\\
BRAZIL\\
\texttt{mateusrodrigues@alumni.usp.br} and \texttt{tojeiro@icmc.usp.br}

\end{document}